\documentclass[12pt]{amsart}
\usepackage[latin1]{inputenc}
\usepackage{color}
\usepackage{bbm}
\usepackage{amsmath,amssymb}

\newtheorem{thm}{Theorem}[section]

\newtheorem{lem}[thm]{Lemma}
\newtheorem{prop}[thm]{Proposition}
\theoremstyle{definition}
\newtheorem{defin}[thm]{Definition}
\newtheorem{rem}[thm]{Remark}

\numberwithin{equation}{section}

\makeatletter
\newcommand{\bN}{\mathbb{N}}

\newcommand{\lin}{\operatorname{span}}
\newcommand{\supp}{\operatorname{supp}}
\newcommand{\dist}{\operatorname{dist}}

\newcommand{\dif}{\,\mathrm{d}}

\newcommand{\cT}{\mathcal{T}}

\newcommand{\cM}{\mathcal{M}}

\newcommand{\bC}{\mathbb{C}}

\DeclareMathOperator{\polyfun}{n}
\newcommand{\charfun}{\ensuremath{\mathbbm 1}}

\DeclareMathOperator{\card}{card}

\allowdisplaybreaks

\begin{document}
\title{Unconditionality of orthogonal spline systems in $L^p$}
\author[M. Passenbrunner]{Markus Passenbrunner}
\address{Institute of Analysis, Johannes Kepler University Linz, Austria, 4040 Linz, Altenberger Strasse 69}
\email{markus.passenbrunner@jku.at}
\keywords{orthonormal spline system, unconditional basis, $L^p$}
\subjclass[2010]{42C10, 46E30}
\date{\today}
\begin{abstract}
Given any natural number $k$ and any dense point sequence $(t_n)$, we prove that the corresponding orthonormal spline system of order $k$ is an unconditional basis in reflexive $L^p$.
\end{abstract}
\maketitle

\section{Introduction}
In this work, we are concerned with orthonormal spline systems of arbitrary order $k$ with arbitrary partitions. We let $(t_n)_{n=2}^\infty$ be a dense sequence of points in the open unit interval such that each point occurs at most $k$ times. Moreover, define $t_0:=0$ and $t_1:=1$. Such point sequences are called \emph{admissible}. For $n\geq 2$, we define $\mathcal S_n^{(k)}$ to be the space of polynomial splines of order $k$ with grid points $(t_j)_{j=0}^n$, where the points $0$ and $1$ both have multiplicity $k$. For each $n\geq 2$, the space $\mathcal S_{n-1}^{(k)}$ has codimension $1$ in $\mathcal S_{n}^{(k)}$ and, therefore, there exists a function $f_{n}^{(k)}\in \mathcal S_{n}^{(k)}$ that is orthonormal to the space $\mathcal S_{n-1}^{(k)}$. Observe that this function $f_{n}^{(k)}$ is unique up to sign. In addition, let $(f_{n}^{(k)})_{n=-k+2}^1$ be the collection of orthonormal polynomials in $L^2[0,1]$ such that the degree of $f_n^{(k)}$ is $k+n-2$. The system of functions $(f_{n}^{(k)})_{n=-k+2}^\infty$ is called \emph{orthonormal spline system of order $k$ corresponding to the sequence  $(t_n)_{n=0}^\infty$}. We will frequently omit the parameter $k$ and write $f_n$ instead of $f_{n}^{(k)}$. The purpose of this article is to prove the following
\begin{thm}\label{thm:uncond}
Let $k\in\mathbb{N}$ and $(t_n)_{n\geq 0}$ be an admissible sequence of knots in $[0,1]$. Then the corresponding general orthonormal spline system of order $k$ is an unconditional basis in $L^p[0,1]$ for every $1<p<\infty$.
\end{thm}
A celebrated result of A. Shadrin \cite{Shadrin2001} states that the orthogonal projection operator onto the space $\mathcal S_n^{(k)}$ is bounded on $L^\infty[0,1]$ by a constant that depends only on the spline order $k$.
As a consequence, $(f_n)_{n\geq -k+2}$ is a basis in $L^p[0,1],$ $1\leq p<\infty$.  
There are various results on the unconditionality of spline systems restricting either the spline order $k$ or the partition $(t_n)_{n\geq 0}$. The first result in this direction is \cite{Bockarev1975}, who proves  that the classical Franklin system---that is orthonormal spline systems of order $2$ corresponding to dyadic knots---is an unconditional basis in $L^p[0,1],\ 1<p<\infty$. This argument was extended in \cite{Ciesielski1975} to prove unconditionality of orthonormal spline systems of arbitrary order, but still restricted to dyadic knots. Considerable effort has been made in the past to weaken the restriction to dyadic knot sequences. In the series of papers \cite{GevorkyanKamont1998, GevorkyanSahakian2000, GevKam2004} this restriction was removed step-by-step for general Franklin systems, with the final result that it was shown for each admissible point sequence $(t_n)_{n\geq 0}$ with parameter $k=2$, the associated general Franklin system forms an unconditional basis in $L^p[0,1]$, $1<p<\infty$. We combine the methods used in \cite{GevorkyanSahakian2000, GevKam2004} with some new inequalities from \cite{PassenbrunnerShadrin2013} to prove that orthonormal spline systems are unconditional in $L^p[0,1],$ $1<p<\infty$, for any spline order $k$ and any admissible point sequence $(t_n)_{n\geq 0}$.

The organization of the present article is as follows. In Section \ref{sec:prel}, we give some preliminary results 
concerning polynomials and splines. Section \ref{sec:proporth} develops some estimates for the orthonormal spline functions $f_n$ using the crucial notion of associating to each function $f_n$ a characteristic interval $J_n$ in a delicate way. Section \ref{sec:comb} treats a central combinatorial result concerning the cardinality of indices $n$ such that a given grid interval $J$ can be a characteristic interval of $f_n$. In Section \ref{sec:techn} we prove a few technical lemmata used in the proof of Theorem \ref{thm:uncond} and Section \ref{sec:main} finally proves Theorem \ref{thm:uncond}. 
We remark that the results and proofs in Sections \ref{sec:techn} and \ref{sec:main} follow closely \cite{GevKam2004}. 
\section{Preliminaries}\label{sec:prel}
Let $k$ be a positive integer. The parameter $k$ will always be used for the order of the underlying polynomials or splines. We use the notation $A(t)\sim B(t)$ to indicate the existence of two constants $c_1,c_2>0$ that depend only on $k$, such that $c_1 B(t)\leq A(t)\leq c_2 B(t)$ for all $t$, where $t$ denotes all implicit and explicit dependences that the expressions $A$ and $B$ might have. If the constants $c_1,c_2$ depend on an additional parameter $p$, we write this as $A(t)\sim_p B(t)$. Correspondingly, we use the symbols $\lesssim,\gtrsim,\lesssim_p,\gtrsim_p$.  For a subset $E$ of the real line, we denote by $|E|$ the Lebesgue measure of $E$ and by $\charfun_E$ the characteristic function of $E$.

First, we recall a few elementary properties of polynomials.
\begin{prop}\label{prop:poly}
Let $0<\rho<1$. Let $I$ be an interval and $A\subset I$ be a subset of $I$ with $|A|\geq \rho |I|$. Then, for every polynomial $Q$ of order $k$ on $I$,
\[
\max_{t\in I}|Q(t)|\lesssim_\rho \sup_{t\in A}|Q(t)|\qquad\text{and}\qquad \int_I |Q(t)|\dif t\lesssim_\rho \int_A |Q(t)|\dif t.
\]
\end{prop}
\begin{lem}\label{lem:proj}
Let $V$ be an open interval and $f$ be a function satisfying $\int_V |f(t)|\dif t\leq \lambda |V|$ for some $\lambda>0$. Then, denoting by $T_V f$ the orthogonal projection of the function $f\cdot\charfun_V$ onto the space of polynomials of order $k$ on $V$,
\begin{equation}\label{eq:polyproj1}
\|T_V f\|_{L^2(V)}^2\lesssim \lambda^2 |V|.
\end{equation}
Moreover,
\begin{equation}\label{eq:polyproj2}
\|T_V f\|_{L^p(V)}\lesssim \|f\|_{L^p(V)},\qquad 1\leq p\leq \infty.
\end{equation}
\end{lem}
\begin{proof}
Let $l_j$, $0\leq j\leq k-1$ be the $j$-th Legendre polynomial on $[-1,1]$ with the normalization $l_j(1)=1$. It is a consequence of the integral identity
\[
l_j(x)=\frac{1}{\pi}\int_0^\pi \Big(x+\sqrt{x^2-1}\cos\varphi\Big)^j\dif \varphi,\quad x\in\bC\setminus\{-1,1\},
\]
that $l_j$ is uniformly bounded by $1$ on the interval $[-1,1]$. We have the orthogonality relation
\begin{equation}\label{eq:legendre}
\int_{-1}^1 l_i(x) l_j(x) \dif x=\frac{2}{2j+1}\delta(i,j),\qquad 0\leq i,j\leq k-1,
\end{equation}
where $\delta(\cdot,\cdot)$ denotes the Kronecker delta. Now let $\alpha:=\inf V$ and $\beta:=\sup V$. For
\[
l_j^V(x):=2^{1/2}|V|^{-1/2} l_j\Big(\frac{2x-\alpha-\beta}{\beta-\alpha}\Big),\quad x\in [\alpha,\beta],
\]
relation \eqref{eq:legendre} still holds for the sequence $(l_j^V)_{j=0}^{k-1}$, that is
\[
\int_\alpha^\beta l_i^V(x) l_j^V(x) \dif x=\frac{2}{2j+1}\delta(i,j),\qquad 0\leq i,j\leq k-1.
\]
So, $T_Vf$ can be represented in the form
\[
T_Vf=\sum_{j=0}^{k-1} \frac{2j+1}{2}\langle f,l_j^V\rangle l_j^V.
\]
Thus we obtain
\begin{align*}
\|T_V f\|_{L^2(V)}&\leq \sum_{j=0}^{k-1}\frac{2j+1}{2}|\langle f,l_j^V \rangle| \|l_j^V\|_{L^2(V)}=\sum_{j=0}^{k-1}\sqrt{\frac{2j+1}{2}}|\langle f,l_j^V\rangle| \\
&\leq \|f\|_{L^1(V)}\sum_{j=0}^{k-1}\sqrt{\frac{2j+1}{2}}\|l_j^V\|_{L^\infty(V)}\lesssim \|f\|_{L^1(V)}|V|^{-1/2},
\end{align*}
Now, \eqref{eq:polyproj1} is a consequence of the assumption $\int_V|f(t)|\dif t\leq \lambda|V|$. If we set $p'=p/(p-1)$, the second inequality \eqref{eq:polyproj2} follows from
\begin{align*}
\|T_Vf\|_{L^p(V)}\leq \sum_{j=0}^{k-1}\frac{2j+1}{2}\|f\|_{L^p(V)}\|l_j^V\|_{L^{p'}(V)}\|l_j^V\|_{L^p(V)}\lesssim \|f\|_{L^p(V)},
\end{align*}
since $\|l_j^V\|_{L^p(V)}\lesssim |V|^{1/p-1/2}$ for $0\leq j\leq k-1$ and $1\leq p\leq \infty$.
\end{proof}

We now let
\begin{equation}\label{eq:part}
\mathcal T=(0=\tau_1=\dots=\tau_k<\tau_{k+1}\leq\dots\leq\tau_M<\tau_{M+1}=\dots=\tau_{M+k}=1)
\end{equation}
be a partition of $[0,1]$ consisting of knots of multiplicity at most $k$, that means $\tau_i<\tau_{i+k}$ for all $1\leq i\leq M$. Let $\mathcal{S}_{\mathcal T}^{(k)}$ be the space of polynomial splines of order $k$ with knots $\mathcal T$.  The basis of $L^\infty$-normalized B-spline functions in $\mathcal{S}_{\mathcal{T}}^{(k)}$ is denoted by $(N_{i,k})_{i=1}^M$ or for short $(N_{i})_{i=1}^M$. Corresponding to this basis, there exists a biorthogonal basis of $\mathcal{S}_{\mathcal{T}}^{(k)}$, which is denoted by $(N_{i,k}^*)_{i=1}^M$ or $(N_{i}^*)_{i=1}^M$.
Moreover, we write $\nu_i = \tau_{i+k}-\tau_i$. We continue with recalling a few important results for B-splines $N_i$ and their dual functions $N_i^*$.

\begin{prop}\label{prop:lpstab} 
Let $1\leq p\leq \infty$ and $g=\sum_{j=1}^M a_j N_j$. Then,
\begin{equation}\label{eq:lpstab}
|a_j|\lesssim |J_j|^{-1/p}\|g\|_{L^p(J_j)},\qquad 1\leq j\leq M,
\end{equation}
where $J_j$ is the subinterval $[\tau_i,\tau_{i+1}]$ of $[\tau_j,\tau_{j+k}]$ of maximal length. Additionally,
\begin{equation}\label{eq:deboorlpstab}
\|g\|_p\sim \Big(\sum_{j=1}^M |a_j|^p \nu_j\Big)^{1/p}=\| (a_j\nu_j^{1/p})_{j=1}^M\|_{\ell^p}.
\end{equation}
Moreover, if $h=\sum_{j=1}^M b_j N_j^*$,
\begin{equation}
\|h\|_p\lesssim\Big(\sum_{j=1}^M |a_j|^p \nu_j^{1-p}\Big)^{1/p}=\|(a_j\nu_j^{1/p-1})_{j=1}^M\|_{\ell^p}.
\label{eq:lpstabdual}
\end{equation}
\end{prop}

The two inequalites \eqref{eq:lpstab} and \eqref{eq:deboorlpstab} are Lemma 4.1 and Lemma 4.2 in \cite[Chapter 5]{DeVoreLorentz1993}, respectively. Inequality \eqref{eq:lpstabdual} is a consequence of the celebrated result of Shadrin \cite{Shadrin2001}, that the orthogonal projection operator onto $\mathcal S_{\mathcal{T}}^{(k)}$ is bounded on $L^\infty$ independently of $\mathcal T$. For a deduction of \eqref{eq:lpstabdual} from this result, see \cite[Property P.7]{Ciesielski2000}.

The next thing to consider are estimates for the inverse of the Gram matrix $(\langle N_{i,k},N_{j,k}\rangle)_{i,j=1}^{M}$. Before we do that, we recall the concept of totally positive matrices:
Let $Q_{m,n}$ the set of strictly increasing sequences of $m$ integers from the set $\{1,\dots,n\}$
and $A$ be an $n\times n$-matrix. For $\alpha,\beta\in Q_{m,n}$, we denote by $A[\alpha;\beta]$ the submatrix of $A$ consisting of the rows indexed by $\alpha$ and the columns indexed by $\beta$. Furthermore we let $\alpha'$ (the complement of $\alpha$) be the uniquely determined element of $Q_{n-m,n}$ that consists of all integers in $\{1,\dots,n\}$ not occurring in $\alpha$. In addition, we use the notation $A(\alpha;\beta):=A[\alpha';\beta']$. 
\begin{defin}
Let $A$ be an $n\times n$-matrix. $A$ is called \emph{totally positive}, if 
\begin{equation*}
\det A[\alpha;\beta]\geq 0,\quad \text{for }\alpha,\beta\in Q_{m,n},1\leq m\leq n.
\end{equation*}
\end{defin}
The cofactor formula $b_{ij}=(-1)^{i+j}\det A(j;i)/\det A$ for the inverse $B=(b_{ij})_{i,j=1}^M$ of the matrix $A$ leads to
\begin{prop}\label{prop:checkerboard}
Inverses $B=(b_{ij})$ of totally positive matrices $A=(a_{ij})$ have the checkerboard property. This means that
\begin{equation*}
(-1)^{i+j} b_{ij}\geq 0\quad \text{for all }i,j.
\end{equation*}
\end{prop}

\begin{thm}[\cite{deBoor1968}]
Let $k\in\mathbb{N}$ and $\mathcal{T}$ be an arbitrary partition of $[0,1]$ as in \eqref{eq:part}. Then the Gram matrix $A=(\langle N_{i,k},N_{j,k}\rangle)_{i,j=1}^{M}$ of the B-spline functions is totally positive.
\end{thm}
This theorem is a consequence of the so called basic composition formula \cite[Chapter 1, Equation (2.5)]{Karlin1968} and the fact that the kernel $N_{i,k}(x)$, depending on the variables $i$ and $x$, is totally positive \cite[Theorem 4.1, Chapter 10]{Karlin1968}. As a consequence, the inverse $B=(b_{ij})_{i,j=1}^M$ of $A$ possesses the checkerboard property by Proposition \ref{prop:checkerboard}.

\begin{thm}[\cite{PassenbrunnerShadrin2013}]\label{thm:maintool}
Let $k\in\mathbb{N}$, the partition $\mathcal T$ be defined as in \eqref{eq:part} and $(b_{ij})_{i,j=1}^M$ the inverse of the Gram matrix $(\langle N_{i,k},N_{j,k}\rangle)_{i,j=1}^{M}$ of B-spline functions $N_{i,k}$ of order $k$ corresponding to the partition $\mathcal T$. Then, 
\[
|b_{ij}|\leq C \frac{\gamma^{|i-j|}}{\tau_{\max(i,j)+k}-\tau_{\min(i,j)}},\qquad 1\leq i,j\leq M,
\]
where the constants $C>0$ and $0<\gamma<1$ depend only on the spline order $k$.
\end{thm}

Let $f\in L^p[0,1]$ for some $1\leq p<\infty$. Since the orthonormal spline system $(f_n)_{n\geq -k+2}$ is a basis in $L^p[0,1]$, we can write $f=\sum_{n=-k+2}^\infty a_n f_n$. Based on this expansion, we define the \emph{square function} $Sf:=\big(\sum_{n=-k+2}^\infty |a_n f_n|^2\big)^{1/2}$ and the \emph{maximal function} $Mf:=\sup_m \big| \sum_{n\leq m} a_n f_n \big|$.
Moreover, given a measurable function $g$, we denote by $\mathcal Mg$ the \emph{Hardy-Littlewood maximal function} of $g$ defined as
\[
\mathcal Mg(x):=\sup_{I\ni x} |I|^{-1} \int_I |g(t)|\dif t,
\]
where the supremum is taken over all intervals $I$ containing the point $x$.

A corollary of Theorem \ref{thm:maintool} gives the following relation between $M$ and $\mathcal M$:
\begin{thm}[\cite{PassenbrunnerShadrin2013}]\label{thm:maxbound}
If $f\in L^1[0,1]$, we have
\[
Mf(t)\lesssim \mathcal M f(t),\qquad t\in[0,1].
\]
\end{thm}

\section{Properties of orthogonal spline functions}\label{sec:proporth}
This section treats the calculation and estimation of one explicit orthonormal spline function $f_n^{(k)}$ for fixed $k\in\mathbb N$ and $n\geq 2$ induced by the admissible sequence $(t_n)_{n=0}^\infty$.
Let $i_0$ be an index with $k+1\leq i_0\leq M$.
The partition $\mathcal T$ is defined as follows:
\begin{align*}
\mathcal T=(0=\tau_1=\dots=\tau_k<\tau_{k+1}&\leq\dots\leq\tau_{i_0} \\
&\leq\dots\leq\tau_{M}<\tau_{M+1}=\dots=\tau_{M+k}=1),
\end{align*}
and the partition $\widetilde{\mathcal T}$ is defined to be the same as $\mathcal T$, but with $\tau_{i_0}$ removed. In the same way we denote by $(N_i:1\leq i\leq M)$ the B-spline functions corresponding to $\mathcal T$ and by $(\widetilde{N}_i:1\leq i\leq M-1)$ the B-spline functions corresponding to $\widetilde{\mathcal T}$. Böhm's formula \cite{Boehm1980} gives us the following relationship between $N_i$ and $\widetilde{N}_i$:
\begin{equation}\label{eq:boehm}
\left\{
\begin{aligned}
\widetilde{N}_i(t)&=N_i(t)  &\text{if }1\leq i\leq i_0-k-1, \\
\widetilde{N}_i(t)&=\frac{\tau_{i_0}-\tau_i}{\tau_{i+k}-\tau_i}N_i(t)+\frac{\tau_{i+k+1}-\tau_{i_0}}{\tau_{i+k+1}-\tau_{i+1}}N_{i+1}(t) &\text{if }i_0-k\leq i\leq i_0-1, \\
\widetilde{N}_i(t)&= N_{i+1}(t) &\text{if }i_0\leq i\leq M-1.
\end{aligned}
\right.
\end{equation}

In order to calculate the orthonormal spline function corresponding to the partitions $\widetilde{\mathcal T}$ and $\mathcal T$, we first determine a function $g\in\lin\{N_i:1\leq i\leq M\}$ such that $g\perp \widetilde{N}_j$ for all $1\leq j\leq M-1$. That is, we assume that $g$ is of the form
\[
g=\sum_{j=1}^M \alpha_j N_j^*,
\]
where $(N_j^*:1\leq j\leq M)$ is the biorthogonal system to the functions $(N_i:1\leq i\leq M)$.
In order for $g$ to be orthogonal to $\widetilde{N}_j$, $1\leq j\leq M-1$, it has to satisfy the identities
\[
0=\langle g,\widetilde{N}_i\rangle=\sum_{j=1}^M \alpha_j\langle N_j^*,\widetilde{N}_i\rangle,\quad 1\leq i\leq M-1.
\]
Using \eqref{eq:boehm}, this implies $\alpha_j=0$ if $1\leq i\leq i_0-k-1$ or $i_0+1\leq i\leq M$. For $i_0-k\leq i\leq i_0-1$, we have the recursion formula
\begin{equation}\label{eq:recalpha}
\begin{aligned}
\alpha_{i+1}\frac{\tau_{i+k+1}-\tau_{i_0}}{\tau_{i+k+1}-\tau_{i+1}}+\alpha_{i}\frac{\tau_{i_0}-\tau_i}{\tau_{i+k}-\tau_i}=0,
 \end{aligned}
\end{equation}
which determines the sequence $(\alpha_j)$ up to a multiplicative constant. We choose 
\begin{equation*}
\alpha_{i_0-k}=\prod_{\ell=i_0-k+1}^{i_0-1}\frac{\tau_{\ell+k}-\tau_{i_0}}{\tau_{\ell+k}-\tau_{\ell}}
\end{equation*}
for symmetry reasons. This starting value and the recursion \eqref{eq:recalpha} yield the explicit formula
\begin{equation}\label{eq:alpha2}
\alpha_j=(-1)^{j-i_0+k}\Big(\prod_{\ell=i_0-k+1}^{j-1}\frac{\tau_{i_0}-\tau_{\ell}}{\tau_{\ell+k}-\tau_{\ell}}\Big)\Big(\prod_{\ell=j+1}^{i_0-1}\frac{\tau_{\ell+k}-\tau_{i_0}}{\tau_{\ell+k}-\tau_{\ell}}\Big),\quad i_0-k\leq j\leq i_0.
\end{equation}
So, the function $g$ is given by
\begin{align*}
g&=\sum_{j=i_0-k}^{i_0} \alpha_j N_j^* =\sum_{j=i_0-k}^{i_0} \sum_{\ell=1}^{M} \alpha_j b_{j\ell} N_\ell,
\end{align*}
where $(b_{j\ell})_{j,\ell=1}^M$ is the inverse of the Gram matrix $(\langle N_j,N_\ell\rangle)_{j,\ell=1}^M$.
We remark that the sequence $(\alpha_j)$ alternates in sign and since the matrix $(b_{j\ell})_{j,\ell=1}^M$ is checkerboard, we see that the B-spline coefficients of $g$, namely
\begin{equation}\label{eq:defwj}
w_\ell:=\sum_{j=i_0-k}^{i_0} \alpha_j b_{j\ell},\qquad 1\leq \ell\leq M,
\end{equation}
satisfy
\begin{equation}\label{eq:betragreinziehen}
\Big| \sum_{j=i_0-k}^{i_0}\alpha_j b_{j\ell}\Big|= \sum_{j=i_0-k}^{i_0}|\alpha_j b_{j\ell}|,\qquad 1\leq j\leq M.
\end{equation}

In the following Definition \ref{def:characteristic}, we assign to each orthonormal spline function a characteristic interval that is a grid point interval $[\tau_i,\tau_{i+1}]$ and lies in the proximity of the newly inserted point $\tau_{i_0}$. We will later see that the choice of this interval is crucial for proving important properties that are needed for showing that the system $(f_n^{(k)})_{n=-k+2}^\infty$ is an unconditional basis in $L^p$, $1<p<\infty$ for all admissible knot sequences $(t_n)_{n\geq 0}$. This approach was already used by G.\ G.\ Gevorkyan and A.\ Kamont \cite{GevKam2004} in the proof that general Franklin systems are unconditional in $L^p$, $1<p<\infty$, where the characteristic intervals were called J-intervals. Since we give a slightly different construction here, we name them characteristic intervals.
\begin{defin}\label{def:characteristic}
Let $\mathcal T,\widetilde{\mathcal T}$ be as above and $\tau_{i_0}$ the new point in $\mathcal T$ that is not present in $\widetilde{\mathcal T}$. We define the \emph{characteristic interval $J$ corresponding to $\tau_{i_0}$} as follows. 
\begin{enumerate}
\item 
Let 
\[
\Lambda^{(0)}:=\{i_0-k\leq j\leq i_0 : |[\tau_j,\tau_{j+k}]|\leq 2\min_{i_0-k\leq \ell\leq i_0}|[\tau_\ell,\tau_{\ell+k}]| \}
\]
be the set of all indices $j$ for which the corresponding support of the B-spline function $N_j$ is approximately minimal. Observe that $\Lambda^{(0)}$ is nonempty.
\item Define
\[
\Lambda^{(1)}:=\{j\in \Lambda^{(0)}: |\alpha_j|=\max_{\ell\in \Lambda^{(0)}} |\alpha_\ell|\}.
\]
For an arbitrary, but fixed index $j^{(0)}\in \Lambda^{(1)}$, set $J^{(0)}:=[\tau_{j^{(0)}},\tau_{j^{(0)}+k}]$.
\item The interval $J^{(0)}$ can now be written as the union of $k$ grid intervals
\[
J^{(0)}=\bigcup_{\ell=0}^{k-1}[\tau_{j^{(0)}+\ell},\tau_{j^{(0)}+\ell+1}]\qquad\text{with }j^{(0)}\text{ as above}.
\]
We define the \emph{characteristic interval} $J=J(\tau_{i_0})$ to be one of the above $k$ intervals that has maximal length.
\end{enumerate}
\end{defin}
We remark that in the definition of $\Lambda^{(0)}$, we may replace the factor $2$ by any other constant $C>1$. It is essential though that $C>1$ in order to obtain the following theorem which is crucial for further investigations.

\begin{thm}\label{thm:estwj}
With the above definition \eqref{eq:defwj} of $w_\ell$ for $1\leq \ell\leq M$ and the index $j^{(0)}$ given in Definition \ref{def:characteristic}, 
\begin{equation}\label{eq:estwj}
|w_{j^{(0)}}|\gtrsim b_{j^{(0)},j^{(0)}}.
\end{equation}
\end{thm}
Before we start the proof of this theorem, we state a few remarks and lemmata.
For the choice of $j^{(0)}$ in Definition \ref{def:characteristic}, we have, by construction, the following inequalities: for all $i_0-k\leq \ell\leq i_0$ with $\ell\neq j^{(0)}$,
\begin{equation}\label{eq:constrJ1}
|\alpha_\ell|\leq |\alpha_{j^{(0)}}|\quad\text{or}\quad |[\tau_\ell,\tau_{\ell+k}]|>2 \min_{i_0-k\leq s\leq i_0} |[\tau_s,\tau_{s+k}]|.
\end{equation}
We recall the identity
\begin{equation}\label{eq:formalphaj}
|\alpha_j|=\Big(\prod_{\ell=i_0-k+1}^{j-1}\frac{\tau_{i_0}-\tau_{\ell}}{\tau_{\ell+k}-\tau_{\ell}}\Big)\Big(\prod_{\ell=j+1}^{i_0-1}\frac{\tau_{\ell+k}-\tau_{i_0}}{\tau_{\ell+k}-\tau_{\ell}}\Big),\qquad i_0-k\leq j\leq i_0.
\end{equation}
Since by \eqref{eq:betragreinziehen},
\begin{equation*}
|w_{j^{(0)}}|=\sum_{j=i_0-k}^{i_0}|\alpha_j b_{j,j^{(0)}}|\geq |\alpha_{j^{(0)}}||b_{j^{(0)},j^{(0)}}|,
\end{equation*}
in order to show \eqref{eq:estwj}, we prove the inequality
\begin{equation*}
|\alpha_{j^{(0)}}|\geq D_k>0
\end{equation*}
with a constant $D_k$ only depending on $k$. 
By \eqref{eq:formalphaj}, this inequality follows from the more elementary inequalities
\begin{equation}\label{eq:alphabed}
\begin{aligned}
\tau_{i_0}-\tau_{\ell}&\gtrsim \tau_{\ell+k}-\tau_{i_0},&\qquad i_0-k+1&\leq \ell\leq j^{(0)}-1, \\
\tau_{\ell+k}-\tau_{i_0}&\gtrsim \tau_{i_0}-\tau_{\ell},&\qquad j^{(0)}+1&\leq \ell\leq i_0-1.
\end{aligned}
\end{equation}
We will only prove the second line of \eqref{eq:alphabed} for all choices of $j^{(0)}$. The first line of \eqref{eq:alphabed} is then proved by a similar argument. We observe that if $j^{(0)}\geq i_0-1$, then there is nothing to prove, so we assume 
\begin{equation}\label{eq:assj0}
j^{(0)}\leq i_0-2.
\end{equation} 
Moreover, we need only show the single inequality
\begin{equation}\label{eq:ulttoshow}
\tau_{j^{(0)}+k+1}-\tau_{i_0}\gtrsim \tau_{i_0}-\tau_{j^{(0)}+1},
\end{equation}
since if we assume \eqref{eq:ulttoshow}, for any $j^{(0)}+1\leq \ell\leq i_0-1$,
\begin{equation*}
\tau_{\ell+k}-\tau_{i_0}\geq \tau_{j^{(0)}+k+1}-\tau_{i_0}\gtrsim \tau_{i_0}-\tau_{j^{(0)}+1}\geq \tau_{i_0}-\tau_\ell.
\end{equation*}

We now choose the index $j$ to be the minimal index in the range $i_0\geq j>j^{(0)}$ such that 
\begin{equation}\label{eq:bedalpha}
|\alpha_j|\leq |\alpha_{j^{(0)}}|.
\end{equation} 
If there is no such $j$, we set $j=i_0+1$. 

If $j\leq i_0$, we employ \eqref{eq:formalphaj} to get that \eqref{eq:bedalpha} is equivalent to
\begin{equation}\label{eq:finalpha}
\begin{aligned}
(\tau_{j+k}-\tau_j)^{1-\delta(j,i_0)}&\prod_{\ell=j^{(0)}\vee(i_0-k+1)}^{j-1}(\tau_{i_0}-\tau_\ell) \\
&\leq(\tau_{j^{(0)}+k}-\tau_{j^{(0)}})^{1-\delta(j^{(0)},i_0-k)}\prod_{\ell=j^{(0)}+1}^{j\wedge(i_0-1)}(\tau_{\ell+k}-\tau_{i_0}),
\end{aligned}
\end{equation}
where $\delta(\cdot,\cdot)$ is the Kronecker delta.
Furthermore, let the index $m$ in the range $i_0-k\leq m\leq i_0$ be such that $\tau_{m+k}-\tau_m=\min_{i_0-k\leq s\leq i_0}(\tau_{s+k}-\tau_s)$.

Now, from the minimality of $j$ and \eqref{eq:constrJ1}, we obtain the inequalities
\begin{equation}\label{eq:condmax}
\tau_{\ell+k}-\tau_\ell>2 (\tau_{m+k}-\tau_m),\qquad j^{(0)}+1\leq \ell\leq j-1.
\end{equation}
Thus, by definition, the index $m$ satisfies
\begin{equation}\label{eq:rangem}
m\leq j^{(0)}\quad\text{or}\quad m\geq j.
\end{equation}
\begin{lem}\label{lem:tau}
In the above notation, if $m\leq j^{(0)}$ and $j-j^{(0)}\geq 2$, we have \eqref{eq:ulttoshow} or more precisely,
\begin{equation}\label{eq:lem:tau}
\tau_{j^{(0)}+k+1}-\tau_{i_0}\geq \tau_{i_0}-\tau_{j^{(0)}+1}.
\end{equation}
\end{lem}
\begin{proof}
We expand the left hand side of \eqref{eq:lem:tau} and write
\[
\tau_{j^{(0)}+k+1}-\tau_{i_0}=\tau_{j^{(0)}+k+1}-\tau_{j^{(0)}+1}-(\tau_{i_0}-\tau_{j^{(0)}+1}).
\]
By \eqref{eq:condmax} (observe that $j-j^{(0)}\geq 2$), we further conclude 
\[
\tau_{j^{(0)}+k+1}-\tau_{i_0}\geq 2(\tau_{m+k}-\tau_m)-(\tau_{i_0}-\tau_{j^{(0)}+1}).
\]
Since $m+k\geq i_0$ and $m\leq j^{(0)}$, we obtain finally
\[
\tau_{j^{(0)}+k+1}-\tau_{i_0}\geq \tau_{i_0}-\tau_{j^{(0)}+1},
\]
which is the conclusion of the lemma.
\end{proof}
\begin{lem}\label{lem:tau2}
Let $j^{(0)},j$ and $m$ be as above. If $j^{(0)}+1\leq \ell\leq j-1$ and $m\geq j$, we have
\begin{equation*}
\tau_{i_0}-\tau_\ell\geq \tau_{\ell+1+k}-\tau_{i_0}.
\end{equation*}
\end{lem}
\begin{proof}
Let $j^{(0)}+1\leq \ell\leq  j-1$. Then we obtain from \eqref{eq:condmax}
\begin{equation}\label{eq:mgtrj1}
\tau_{i_0}-\tau_{\ell}=\tau_{\ell+1+k}-\tau_{\ell}-(\tau_{\ell+1+k}-\tau_{i_0})\geq 2(\tau_{m+k}-\tau_m)-(\tau_{\ell+1+k}-\tau_{i_0}).
\end{equation}
Since we assumed $m\geq j\geq \ell+1$, we get $m+k\geq \ell+1+k$ and additionally we have $m\leq i_0$ by definition of $m$. Thus we conclude from \eqref{eq:mgtrj1}
\begin{equation*}
\tau_{i_0}-\tau_{\ell}\geq \tau_{\ell+1+k}-\tau_{i_0}.
\end{equation*}
Since the index $\ell$ was arbitrary in the range $j^{(0)}+1\leq \ell\leq  j-1$, the proof of the lemma is completed.
\end{proof}

\begin{proof}[Proof of Theorem \ref{thm:estwj}]
We employ the above definition of the indices $j^{(0)},j,$ and $m$ and split our analysis in a few cases distinguishing various possibilities for the parameters $j^{(0)}$ and $j$. In each case we will show \eqref{eq:ulttoshow}.

\textsc{Case 1: }There is no index $j>j^{(0)}$ such that $|\alpha_j|\leq |\alpha_{j^{(0)}}|$. \\
In this case, \eqref{eq:rangem} implies $m\leq j^{(0)}$. Since $j^{(0)}\leq i_0-2$ by \eqref{eq:assj0}, 
 we apply Lemma \ref{lem:tau} to conclude the proof of \eqref{eq:ulttoshow}.

\textsc{Case 2: } $i_0-k+1\leq j^{(0)}<j\leq i_0-1$. \\
Using the restrictions on our parameters $j^{(0)}$ and $j$, we see that \eqref{eq:finalpha} becomes
\begin{equation*}
(\tau_{j^{(0)}+k}-\tau_{j^{(0)}})\prod_{\ell=j^{(0)}+1}^{j}(\tau_{\ell+k}-\tau_{i_0})\geq (\tau_{j+k}-\tau_j)\prod_{\ell=j^{(0)}}^{j-1}(\tau_{i_0}-\tau_{\ell}).
\end{equation*}
This implies
\begin{align*}
	\tau_{j^{(0)}+k+1}-\tau_{i_0}\geq \frac{(\tau_{j+k}-\tau_j)(\tau_{i_0}-\tau_{j^{(0)}})}{\tau_{j^{(0)}+k}-\tau_{j^{(0)}}}\prod_{\ell=j^{(0)}+1}^{j-1}\frac{\tau_{i_0}-\tau_\ell}{\tau_{\ell+1+k}-\tau_{i_0}}.
\end{align*}
Since by definition of $j^{(0)}$, we have in particular $\tau_{j^{(0)}+k}-\tau_{j^{(0)}}\leq 2 (\tau_{j+k}-\tau_j)$, we conclude further 
\begin{equation}\label{eq:case2-1}
\tau_{j^{(0)}+k+1}-\tau_{i_0}\geq \frac{\tau_{i_0}-\tau_{j^{(0)}+1}}{2}\prod_{\ell=j^{(0)}+1}^{j-1}\frac{\tau_{i_0}-\tau_\ell}{\tau_{\ell+1+k}-\tau_{i_0}}.
\end{equation}

If $j=j^{(0)}+1$, the assertion \eqref{eq:ulttoshow} follows from \eqref{eq:case2-1}, since then, the product is empty.

If $j\geq j^{(0)}+2$ and $m\leq j^{(0)}$, we apply Lemma \ref{lem:tau} to obtain \eqref{eq:ulttoshow}.

If $j\geq j^{(0)}+2$ and $m\geq j$, we use Lemma \ref{lem:tau2} on the terms in the product appearing in \eqref{eq:case2-1} to conclude \eqref{eq:ulttoshow}.

This finishes the proof of Case 2.

\textsc{Case 3: }$i_0-k+1\leq j^{(0)}<j=i_0$.\\
Recall that $j^{(0)}\leq i_0-2=j-2$ by \eqref{eq:assj0}. If $m\leq j^{(0)}$, we apply Lemma \ref{lem:tau} and we are done with the proof of \eqref{eq:ulttoshow}.
So we assume $m\geq j$. Since $i_0=j$ and $m\leq i_0$, we have $m=j$. The restrictions on the indices $j^{(0)},j$ yield that condition \eqref{eq:finalpha} is nothing else than
\begin{equation*}
(\tau_{j^{(0)}+k}-\tau_{j^{(0)}})\prod_{\ell=j^{(0)}+1}^{i_0-1}(\tau_{\ell+k}-\tau_{i_0})\geq \prod_{\ell=j^{(0)}}^{i_0-1}(\tau_{i_0}-\tau_{\ell}).
\end{equation*}
Thus, in order to show \eqref{eq:ulttoshow}, it is enough to prove that there exists a constant $D_k>0$ only depending on $k$ such that 
\begin{equation}\label{eq:tau1}
\frac{\tau_{i_0}-\tau_{j^{(0)}}}{\tau_{j^{(0)}+k}- \tau_{j^{(0)}}}\prod_{\ell=j^{(0)}+2}^{i_0-1}\frac{\tau_{i_0}-\tau_{\ell}}{\tau_{\ell+k}-\tau_{i_0}}\geq D_k.
\end{equation}
First observe that by Lemma \ref{lem:tau2},
\[
\tau_{i_0}-\tau_{j^{(0)}}\geq \tau_{j^{(0)}+k+2}-\tau_{i_0}\geq \tau_{j^{(0)}+k}-\tau_{i_0}.
\]
Inserting this inequality in \eqref{eq:tau1} and applying Lemma \eqref{lem:tau2} directly to the terms in the product, we obtain the assertion \eqref{eq:tau1}. 

\textsc{Case 4: } $i_0-k=j^{(0)}<j=i_0$. \\
We have $j^{(0)}\leq i_0-2$ by \eqref{eq:assj0}. If  $m\leq j^{(0)}$,  just apply Lemma \ref{lem:tau} to obtain \eqref{eq:ulttoshow}. Thus we assume $m\geq j$. Since $i_0=j$ and $m\leq i_0$, we have $m=j$. The restrictions on the indices $j^{(0)},j$ yield that condition \eqref{eq:finalpha} takes the form
\begin{equation*}
\prod_{\ell=i_0-k+1}^{i_0-1}(\tau_{\ell+k}-\tau_{i_0})\geq \prod_{\ell=i_0-k+1}^{i_0-1}(\tau_{i_0}-\tau_{\ell}).
\end{equation*}
Thus, in order to show \eqref{eq:ulttoshow}, it is enough to prove that there exists a constant $D_k>0$ only depending on $k$ such that 
\begin{equation*}
\prod_{\ell=i_0-k+2}^{i_0-1}\frac{\tau_{i_0}-\tau_{\ell}}{\tau_{\ell+k}-\tau_{i_0}}\geq D_k.
\end{equation*}
But this is a consequence of Lemma \ref{lem:tau2}, finishing the proof of Case 4.

\textsc{Case 5: }$i_0-k=j^{(0)}<j\leq i_0-1$.\\
In this case, \eqref{eq:ulttoshow} becomes
\begin{equation}\label{eq:tautoshow5}
\tau_{i_0+1}-\tau_{i_0}\gtrsim \tau_{i_0}-\tau_{i_0-k+1}.
\end{equation}
and \eqref{eq:finalpha} is nothing else than
\begin{equation}\label{eq:lastcase}
\prod_{\ell=i_0-k+1}^{j} (\tau_{\ell+k}-\tau_{i_0}) \geq (\tau_{j+k}-\tau_j)\prod_{\ell=i_0-k+1}^{j-1}(\tau_{i_0}-\tau_{\ell}).
\end{equation}
For $j=i_0-k+1$, \eqref{eq:tautoshow5} is implied very easily from \eqref{eq:lastcase}. If we assume $j-j^{(0)}\geq 2$ and $m\leq j^{(0)}$, we just apply Lemma \ref{lem:tau} to obtain \eqref{eq:ulttoshow}. 
If $j-j^{(0)}\geq 2$ and $m\geq j$, showing \eqref{eq:tautoshow5} is equivalent to the existence of a constant $D_k>0$ only depending on $k$ such that
\begin{equation*}
\frac{(\tau_{j+k}-\tau_j)\prod_{\ell=i_0-k+2}^{j-1}(\tau_{i_0}-\tau_{\ell})}{\prod_{\ell=i_0-k+2}^{j}(\tau_{\ell+k}-\tau_{i_0})}\geq D_k.
\end{equation*}
This follows from the obvious inequality $\tau_{j+k}-\tau_j\geq \tau_{j+k}-\tau_{i_0}$ and from Lemma \ref{lem:tau2}. Thus, the proof of Case 5 is completed, thereby concluding the proof of Theorem \ref{thm:estwj}.
\end{proof}

We will use this result to prove lemmata connecting the $L^p$ norm of the function $g$ and the corresponding  characteristic interval $J$. Before we start, we need another simple

\begin{lem}\label{lem:estdiaginverse}
Let $C=(c_{ij})_{i,j=1}^n$ be a symmetric positive definite matrix. Then, for $(d_{ij})_{i,j=1}^n=C^{-1}$ we have
\[
c_{ii}^{-1}\leq d_{ii},\qquad 1\leq i\leq n.
\]
\end{lem}
\begin{proof}
Since $C$ is symmetric, $C$ is diagonalizable and we have 
\[
C=S\Lambda S^T,
\]
for some orthogonal matrix $S=(s_{ij})_{i,j=1}^n$ and for the diagonal matrix $\Lambda$ consisting of the eigenvalues $\lambda_1,\dots,\lambda_n$ of $C$. These eigenvalues are positive, since $C$ is positive definite. Clearly,
\[
C^{-1}=S\Lambda ^{-1}S^T.
\]
Let $i$ be an arbitrary integer in the range $1\leq i\leq n$. Then,
\[
c_{ii}=\sum_{\ell=1}^n s_{i\ell}^2 \lambda_\ell\qquad\text{and}\qquad d_{ii}=\sum_{\ell=1}^n s_{i\ell}^2 \lambda_\ell^{-1}.
\]
Since $\sum_{\ell=1}^n s_{i\ell}^2=1$ and the function $x\mapsto x^{-1}$ is convex on $(0,\infty)$, we conclude by Jensen's inequality
\[
c_{ii}^{-1}=\big(\sum_{\ell=1}^n s_{i\ell}^2 \lambda_\ell\big)^{-1}\leq \sum_{\ell=1}^n s_{i\ell}^2 \lambda_\ell^{-1}=d_{ii},
\]
and thus the assertion of the lemma.
\end{proof}

\begin{lem}\label{lem:orthsplineJinterval}Let $\mathcal T,\,\widetilde{\mathcal T}$ be as above and $g=\sum_{j=1}^M w_jN_j$ be the function in $\lin\{N_i:1\leq i\leq M\}$ that is orthogonal to every $\widetilde{N}_i$, $1\leq i\leq M-1$, with $(w_j)_{j=1}^M$ given in \eqref{eq:defwj}. Moreover, let $\varphi=g/\|g\|_2$  be the $L^2$-normalized orthogonal spline function corresponding to the mesh point $\tau_{i_0}$.
Then,
\[
\|\varphi\|_{L^p(J)}\sim\|\varphi\|_p\sim |J|^{1/p-1/2},\qquad 1\leq p\leq \infty,
\]
where $J$ is the characteristic interval associated to the point $\tau_{i_0}$ given in Definition \ref{def:characteristic}.
\end{lem}

\begin{proof}
As a consequence of inequality \eqref{eq:lpstab} in Proposition \ref{prop:lpstab}, we get
\begin{equation}\label{eq:orthlpnorm:1}
\|g\|_{L^p(J)} \gtrsim |J|^{1/p} |w_{j^{(0)}}|.
\end{equation}
By Theorem \ref{thm:estwj}, $|w_{j^{(0)}}|\gtrsim b_{j^{(0)},j^{(0)}}$, where we recall that $(b_{ij})_{i,j=1}^M$ is the inverse of the Gram matrix $(a_{ij})_{i,j=1}^M=(\langle N_i,N_j\rangle)_{i,j=1}^M$. Now we invoke Lemma \ref{lem:estdiaginverse} and identity \eqref{eq:deboorlpstab} of Proposition \ref{prop:lpstab} to conclude from \eqref{eq:orthlpnorm:1}
\begin{align*}
\|g\|_{L^p(J)} & \gtrsim |J|^{1/p} b_{j^{(0)},j^{(0)}}\geq |J|^{1/p} a_{j^{(0)},j^{(0)}}^{-1} \\
				& = |J|^{1/p} \|N_{j^{(0)}}\|^{-2}_2 \gtrsim |J|^{1/p} \nu_{j^{(0)}}^{-1}
\end{align*}
Since, by construction, $J$ is the maximal subinterval of $J^{(0)}$ and there are exactly $k$ subintervals of $J^{(0)}$, we finally get 
\begin{equation}\label{eq:Jinterval1}
\|g\|_{L^p(J)} \gtrsim |J|^{1/p-1}.
\end{equation}

On the other hand, $g=\sum_{j=i_0-k}^{i_0}\alpha_j N_j^*$, so we use equation \eqref{eq:lpstabdual} of Proposition \ref{prop:lpstab} to obtain
\[
\|g\|_p \lesssim \Big(\sum_{j=i_0-k}^{i_0} |\alpha_j|^p \nu_j^{1-p}\Big)^{1/p}.
\]
Since $|\alpha_j|\leq 1$ for all $j$ and $\nu_{j^{(0)}}$ is minimal (up to the factor $2$) among the values $\nu_j,\ i_0-k\leq j\leq i_0$, we can estimate this further by 
\[
\|g\|_p\lesssim \nu_{j^{(0)}}^{1/p-1}.
\]
We now use the inequality $|J|\leq \nu_{j^{(0)}}=|J^{(0)}|$ from the construction of $J$ to get
\begin{equation}\label{eq:Jinterval2}
\|g\|_p \lesssim |J|^{1/p-1}
\end{equation}
The assertion of the lemma now follows from the two inequalites \eqref{eq:Jinterval1} and \eqref{eq:Jinterval2} after renormalization.
\end{proof}
By $d_{\mathcal T} (x)$ we denote the number of points in $\mathcal T$ between $x$ and $J$ counting endpoints of $J$. Correspondingly, for an interval $V\subset [0,1]$, by $d_{\mathcal T}(V)$ we denote the number of points in $\mathcal T$ between $V$ and $J$ counting endpoints of both $J$ and $V$.

\begin{lem}\label{lem:lporthspline}
Let $\mathcal T,\widetilde{\mathcal T}$ be as above and $g=\sum_{j=1}^M w_jN_j$ be orthogonal to every $\widetilde{N}_i$, $1\leq i\leq M-1$, with $(w_j)_{j=1}^M$ as in \eqref{eq:defwj}. Moreover, let $\varphi=g/\|g\|_2$  be the normalized orthogonal spline function corresponding to $\tau_{i_0}$ and $\gamma<1$ the constant from Theorem \ref{thm:maintool} depending only on the spline order $k$.  Then we have
\begin{equation}\label{eq:wj}
|w_j|\lesssim \frac{\gamma^{d_{\mathcal T}(\tau_j)}}{|J|+\dist(\supp N_j,J)+\nu_j}\quad\text{for all }1\leq j\leq M.
\end{equation}
Moreover, if $x<\inf J$, we have
\begin{equation}\label{eq:phiplinks}
\|\varphi\|_{L^p(0,x)}
\lesssim \frac{\gamma^{d_{\mathcal T}(x)}|J|^{1/2}}{(|J|+\dist(x,J))^{1-1/p}},\qquad 1\leq p\leq \infty.
\end{equation}
Similarly, for $x>\sup J$,
\begin{equation}\label{eq:phiprechts}
\|\varphi\|_{L^p(x,1)}
\lesssim \frac{\gamma^{d_{\mathcal T}(x)}|J|^{1/2}}{(|J|+\dist(x,J))^{1-1/p}},\qquad 1\leq p\leq \infty.
\end{equation}
\end{lem}

\begin{proof}
We begin with showing \eqref{eq:wj}.
By definition of $w_j$ and  $\alpha_\ell$ (see \eqref{eq:defwj} and \eqref{eq:alpha2}), we have
\[
|w_j|\lesssim \max_{i_0-k\leq \ell\leq i_0} |b_{j\ell}|.
\]
Now we invoke Theorem \ref{thm:maintool} to deduce
\begin{equation}\label{eq:startwj}
\begin{aligned}
|w_j|
&\lesssim \frac{\max_{i_0-k\leq \ell\leq i_0}\gamma^{|\ell-j|}}{\min_{i_0-k\leq \ell\leq i_0}(\tau_{\max(\ell,j)+k}-\tau_{\min(\ell, j)})} \\
&\lesssim \frac{\gamma^{d_{\mathcal T}(\tau_j)}}{\min_{i_0-k\leq \ell\leq i_0}(\tau_{\max(\ell,j)+k}-\tau_{\min(\ell, j)})},
\end{aligned}
\end{equation}
where the second inequality follows from the location of $J$ in the interval $[\tau_{i_0-k},\tau_{i_0+k}]$. It remains to estimate the minimum in the denominator of this expression. Let $\ell$ be an arbitrary natural number in the range $i_0-k\leq \ell\leq i_0$. First we observe
\begin{equation}
\tau_{\max(\ell,j)+k}-\tau_{\min(\ell,j)}\geq \tau_{j+k}-\tau_j=|\supp N_j|=\nu_j.
\label{eq:phi1}
\end{equation}
Moreover, by definition of $J$,
\begin{equation}
\begin{aligned}
\tau_{\max(\ell,j)+k}-\tau_{\min(\ell,j)} & \geq  \min_{i_0-k\leq r\leq i_0} (\tau_{r+k}-\tau_r) \geq |J^{(0)}|/2\geq  |J|/2.
\end{aligned}
\label{eq:phi2}
\end{equation}
If now $j\geq \ell$,
\begin{equation}
\tau_{\max(\ell,j)+k}-\tau_{\min(\ell,j)}=\tau_{j+k}-\tau_\ell\geq \tau_{j+k}-\tau_{i_0}\geq \max(\tau_j-\sup J^{(0)},0),
\label{eq:phi3}
\end{equation}
since $\tau_{i_0}\leq \sup J^{(0)}$. But $\max(\tau_j-\sup J^{(0)},0)=d([\tau_j,\tau_{j+k}],J^{(0)})$ due to the fact that $\inf J^{(0)}\leq \tau_{i_0}\leq \tau_{\ell+k}\leq \tau_{j+k}$ for the current choice of $j$.
Additionally, $d([\tau_j,\tau_{j+k}],J)\leq |J^{(0)}|+ d([\tau_j,\tau_{j+k}], J^{(0)})$. So, as a consequence of \eqref{eq:phi3},
\begin{equation}
\tau_{\max(\ell,j)+k}-\tau_{\min(\ell,j)}\geq d([\tau_j,\tau_{j+k}],J)-|J^{(0)}|.
\label{eq:phi4}
\end{equation}
An analogous calculation proves \eqref{eq:phi4} also in the case $j\leq \ell$.
We now combine our inequality \eqref{eq:startwj} with \eqref{eq:phi1}, \eqref{eq:phi2} and \eqref{eq:phi4} to obtain the assertion \eqref{eq:wj}.

We now consider the integral $\big(\int_0^x |g(t)|^p\dif t\big)^{1/p}$ for $x<\inf J$. The analogous estimate  \eqref{eq:phiprechts} follows from a similar argument. Let $\tau_s$ be the first grid point in $\mathcal T$ to the right of $x$ and observe that that $\supp N_r\cap [0,\tau_s)=\emptyset$ for $r\geq s$.
Then we get
\[
\|g\|_{L^p(0,x)}\leq \|g\|_{L^p(0,\tau_s)}\leq \Big\|\sum_{i=1}^{s-1} w_i N_i\Big\|_p.
\]
By \eqref{eq:deboorlpstab} of Proposition \ref{prop:lpstab}, we conclude further
\[
\|g\|_{L^p(0,x)}\leq \big\|\big(w_i\nu_i^{1/p}\big)_{i=1}^{s-1}\big\|_{\ell^p}.
\]
We now use \eqref{eq:wj} for $w_i$ and get
\[
\|g\|_{L^p(0,x)}\lesssim \Big\|\Big(\frac{\gamma^{d_{\mathcal T}(\tau_i)}\nu_i^{1/p}}{|J|+\dist(\supp N_i,J)+\nu_i}\Big)_{i=1}^{s-1}\Big\|_{\ell^p}.
\]
Since $\nu_i\leq |J|+\dist(\supp N_i,J)+\nu_i$ for all $1\leq i\leq M$ and $\dist(\supp N_i,J)+\nu_i\geq \dist(x,J)$ for all $1\leq i\leq s-1$, we estimate the last display further to get 
\[
\|g\|_{L^p(0,x)}\lesssim \big(|J|+\dist(x,J)\big)^{-1+1/p}\|(\gamma^{d_{\mathcal T}(\tau_i)})_{i=1}^{s-1}\|_{\ell^p}.
\]
This $\ell^p$-norm is a geometric sum and the biggest term is $\gamma^{d_{\mathcal T}(x)}$, so we obtain
\[
\|g\|_{L^p(0,x)} \lesssim \frac{\gamma^{d_{\mathcal T}(x)}}{(|J|+\dist(x,J))^{1-1/p}}.
\]
This concludes the proof of the lemma, since we have seen in the proof of Lemma \ref{lem:orthsplineJinterval} that $\|g\|_2\sim |J|^{-1/2}$.
\end{proof}
\begin{rem}\label{rem:linftyorthspline}
Analogously we obtain the inequality
\begin{align*}
\sup_{\tau_{j-1}\leq t\leq \tau_j}|\varphi(t)|&\lesssim \max_{j-k\leq i\leq j-1} \frac{\gamma^{d_{\mathcal T}(\tau_i)}|J|^{1/2}}{|J|+\dist(\supp N_i,J)+\nu_i} \\
&\lesssim \frac{\gamma^{d_{\mathcal T}(\tau_{j})}|J|^{1/2}}{|J|+\dist(J,[\tau_{j-1},\tau_j])+|[\tau_{j-1},\tau_j]|},
\end{align*}
since for all integers $i$ with $j-k\leq i\leq j-1$ we have $[\tau_{j-1},\tau_j]\subset \supp N_i$.
\end{rem}
\section{Combinatorics of characteristic intervals}\label{sec:comb}
Let $(t_n)_{n=0}^\infty$ be an admissible sequence of points and $(f_n)_{n=-k+2}^\infty$ the corresponding orthonormal spline functions of order $k$.
For $n\geq 2$, the associated partitions $\mathcal T_n$ to $f_n$ are defined to consist of the grid points $(t_j)_{j=0}^n$, the knots $t_0=0$ and $t_1=1$ having both multiplicity $k$ in $\mathcal T_n$.
If $n\geq 2$, we denote by $\ J_n^{(0)}$ and $J_n$ the characteristic intervals $J^{(0)}$ and $J$ from Definition \ref{def:characteristic} associated to the new grid point $t_n$. If $n$ is in the range $-k+2\leq n\leq 1$, we additionally set $J_n:=[0,1]$. For any $x\in [0,1]$, we define $d_n(x)$ to be the number of grid points in $\mathcal T_n$ between $x$ and $J_n$ counting endpoints of $J_n$. Moreover, for a subinterval $V$ of $[0,1]$, we denote by $d_n(V)$ the number of knots in $\mathcal T_n$ between $V$ and $J_n$ counting endpoints of both $V$ and $J_n$.
Finally, if $\mathcal T_n$ is of the form
\begin{align*}
\mathcal T_n=(0=\tau_{n,1}&=\dots=\tau_{n,k}<\tau_{n,k+1}\leq&\\
&\leq\dots\leq\tau_{n,n+k-1}<\tau_{n,n+k}=\dots=\tau_{n,n+2k-1}=1),
\end{align*}
and if $t_n=\tau_{n,i_0}$, then we denote by $t_n^{+\ell}$ the point $\tau_{n,i_0+\ell}$.

For the proof of the central Lemma \ref{lem:jinterval} of this section, we need the combinatorial Lemma of Erd\H{o}s and Szekeres:
\begin{lem}[Erd\H{o}s-Szekeres]\label{lem:erdsze}
Let $n$ be an integer. Every sequence $(x_1,\dots,$ $x_{(n-1)^2+1})$ of length $(n-1)^2+1$ contains a monotone sequence of length $n$.
\end{lem}

We now use this result to prove a lemma about the combinatorics of characteristic intervals $J_n$:

\begin{lem}\label{lem:jinterval}
Let $x,y\in (t_n)_{n=0}^\infty$ such that $x<y$ and $0\leq \beta\leq 1/2$. Then there exists a constant $F_{k}$ only depending on $k$ such that
\[
N_0:=\card\{n:J_n\subseteq [x,y], |J_n|\geq(1-\beta)|[x,y]|\} \leq F_{k},
\]
where $\card E$ denotes the cardinality of the set $E$.
\end{lem}
\begin{proof}
Let $N_0$ be defined as above.
If $n$ is an index such that $J_n\subseteq [x,y]$ and $|J_n|\geq (1-\beta)|[x,y]|$, then, by definition of $J_n$, we have $t_n\in[0,(1-\beta)x+\beta y]\cup [\beta x+(1-\beta)y,1]$. Thus, by the pigeon hole principle, in  one of the two sets $[0,(1-\beta)x+\beta y]$ and  $[\beta x+(1-\beta)y,1]$, there are at least 
\[
N_1:=\Big\lfloor \frac{N_0-1}{2}\Big\rfloor +1 
\]
indices $n$ with $J_n\subset [x,y]$ and $|J_n|\geq (1-\beta)|[x,y]|$.  Assume without loss of generality that this set is $[\beta x+(1-\beta)y,1]$.  Now, let $(n_i)_{i=1}^{N_1}$ be an increasing sequence of indices such that $t_{n_i}\in [\beta x+(1-\beta)y,1]$ and $J_{n_i}\subset [x,y]$, $ |J_{n_i}|\geq (1-\beta)|[x,y]|$  for every $1\leq i\leq N_1$. Observe that for such $i$, $J_{n_i}$ is to the left of $t_{n_i}$.
 By the Erd\H{o}s-Szekeres-Lemma \ref{lem:erdsze}, the sequence $(t_{n_i})_{i=1}^{N_1}$ contains a monotone subsequence $(t_{m_i})_{i=1}^{N_2}$ of length
\[
N_2:=\lfloor\sqrt{N_1-1}\rfloor +1.
\]

If $(t_{m_i})_{i=1}^{N_2}$ is increasing, we obtain that $N_2\leq k$. Indeed, if $N_2\geq k+1$,
there are at least $k$ points (namely $t_{m_1},\dots,t_{m_k}$) in the sequence $\cT_{m_{k+1}}$ between $\inf J_{m_{k+1}}$ and $t_{m_{k+1}}$. This is in conflict with the location of $J_{m_{k+1}}$.

If $(t_{m_i})_{i=1}^{N_2}$ is decreasing, we let
\[
s_1\leq \dots \leq s_L
\]
be an enumeration of the elements in the sequence $\cT_{m_1}$ such that $\inf J_{m_1}\leq s\leq t_{m_1}$.
By definition of $J_{m_1}$, we obtain that $L\leq k+1$.  Thus, there are at most $k$ intervals $[s_\ell,s_{\ell+1}], 1\leq \ell\leq L-1$, contained in $[\inf J_{m_1},t_{m_1}]$.
 Again, by the pigeon hole principle, there exists one index $1\leq \ell\leq L-1$ such that the interval $[s_\ell,s_{\ell+1}]$ contains (at least)
\[
N_3:=\Big\lfloor \frac{N_2-1}{k}\Big\rfloor +1
\]
points of the sequence $(t_{m_i})_{i=1}^{N_2}$.
 Let $(t_{r_i})_{i=1}^{N_3}$ be a subsequence of length $N_3$ of such points. Furthermore, define
\[
N_4:=\Big\lfloor \frac{N_3}{k}\Big\rfloor.
\]
Since $(t_{r_i})_{i=1}^{N_3}$ is decreasing, we have a quantity of $N_4$ disjoint intervals  
\[
I_\mu:=(t_{r_{\mu\cdot k}},t_{r_{\mu\cdot k}}^{+k})\subseteq [s_\ell,s_{\ell+1}],\qquad 1\leq \mu\leq N_4.
\]
Consequently, there exists (at least) one index $\mu$ such that
\[
|I_\mu|\leq \frac{|[s_\ell,s_{\ell+1}]|}{N_4}.
\]
We next observe that the definition of $J_{m_1}$ yields
\begin{equation*}
|J_{m_1}|\geq |[s_\ell,s_{\ell+1}]|.
\end{equation*}
We thus get
\begin{equation}\label{eq:J02}
\begin{aligned}
|J_{r_{\mu\cdot k}}^{(0)}|&\geq |J_{r_{\mu\cdot k}}|\geq (1-\beta)|[x,y]|\geq (1-\beta)|J_{m_1}|\\
&\geq (1-\beta)|[s_\ell,s_{\ell+1}]|\geq (1-\beta)N_4|I_\mu|.
\end{aligned}
\end{equation}
On the other hand, the construction of $J_{r_{\mu\cdot k}}^{(0)}$ implies in particular
\begin{equation}\label{eq:J01}
|J_{r_{\mu\cdot k}}^{(0)}|\leq 2(t_{r_{\mu\cdot k}}^{+k}-t_{r_{\mu \cdot k}}) =  2|I_\mu|.
\end{equation}
The inequalities \eqref{eq:J02} and \eqref{eq:J01} imply $N_4\leq 2/(1-\beta)\leq 4$. Since the definition of $N_4$ involves only $k$, this proves the assertion of the lemma.
\end{proof}

\section{Technical estimates}\label{sec:techn}
\begin{lem}\label{lem:techn1}
Let $f=\sum_{n=-k+2}^\infty a_n f_n$ and $V$ be an open subinterval of $[0,1]$.
Then,
\begin{align}
\int_{V^c}\sum_{j\in\Gamma}|a_j f_j(t)|\dif t&\lesssim\int_V\Big(\sum_{j\in\Gamma}|a_j f_j(t)|^2\Big)^{1/2}\dif t, \label{eq:lemtechn1:1} 
\end{align}
where
\[
\Gamma:=\{j : J_j\subset  V\text{ and }-k+2\leq j<\infty\}.
\]
\end{lem}
\begin{proof}
First, assume that $|V|=1$. Then \eqref{eq:lemtechn1:1} holds trivially. In the following, we assume that $|V|<1$. We define $x:=\inf V$, $y:=\sup V$ and fix an index $n\in \Gamma$. Observe that in this case, the definition of $\Gamma$ implies $n\geq 2$, since $J_j=[0,1]$ for $-k+2\leq j\leq 1$.
We only estimate the integral in \eqref{eq:lemtechn1:1} over the interval $[y,1]$. The integral over $[0,x]$ is estimated similarly. Lemma \ref{lem:lporthspline} implies
\[
\int_y^1|f_n(t)|\dif t\lesssim \gamma^{d_n(y)}|J_n|^{1/2}.
\]
Applying Lemma \ref{lem:orthsplineJinterval} yields
\begin{equation}\label{eq:techn1:1}
\int_y^1 |f_n(t)|\dif t\lesssim \gamma^{d_n(y)}\int_{J_n}|f_n(t)|\dif t.
\end{equation}
Now choose $\beta=1/4$ and let $J_n^\beta$ be the unique closed interval that satisfies
\[
|J_n^\beta|=\beta|J_n|\quad\text{and}\quad \inf J_n^\beta=\inf J_n.
\]
Since $f_n$ is a polynomial of order $k$ on the interval $J_n$, we apply Proposition \ref{prop:poly} to \eqref{eq:techn1:1} and estimate further
\begin{equation}\label{eq:techn1:2}
\int_y^1 |a_n f_n(t)|\dif  t \lesssim \gamma^{d_n(y)} \int_{J_n^\beta} |a_n f_n(t)|\dif t\leq \gamma^{d_n(y)} \int_{J_n^\beta} \Big(\sum_{j\in\Gamma}|a_jf_j(t)|^2\Big)^{1/2}\dif t
\end{equation}
Define $\Gamma_s:=\{j\in\Gamma:d_j(y)=s\}$ for $s\geq 0$. For fixed  $s\geq 0$ and $j_1,\,j_2\in\Gamma_s$, we have either
\[
J_{j_1}\cap J_{j_2}=\emptyset\quad\text{or}\quad \sup J_{j_1}=\sup J_{j_2}.
\]
So, Lemma \ref{lem:jinterval} implies that there exists a constant $F_{k}$ only depending on $k$, such that each point $t\in V$ belongs to at most $F_{k}$ intervals $J_{j}^\beta$, $j\in\Gamma_s$. Thus, summing over $j\in \Gamma_s$, we get from \eqref{eq:techn1:2}
\begin{equation*}
\begin{aligned}
\sum_{j\in\Gamma_s}\int_y^1 |a_jf_j(t)|\dif t &\lesssim \sum_{j\in\Gamma_s}\gamma^s \int_{J_{j}^\beta} \Big(\sum_{\ell\in\Gamma}|a_\ell f_\ell(t)|^2\Big)^{1/2} \dif t \\
&\lesssim \gamma^s\int_V\Big(\sum_{\ell \in\Gamma}|a_\ell f_\ell(t)|^2\Big)^{1/2} \dif t.
\end{aligned}
\end{equation*}
Finally, we sum over $s\geq 0$ to obtain inequality \eqref{eq:lemtechn1:1}.
\end{proof}
Let $g$ be a real-valued function defined on the closed unit interval. In the following, we denote by  $[g>\lambda]$ the  set $\{x\in[0,1]: g(x)>\lambda\}$ for any number $\lambda>0$.
\begin{lem}\label{lem:triv}
Let $f=\sum_{n=-k+2}^\infty a_n f_n$ with only finitely many nonzero coefficients $a_n$, $\lambda>0,\ r<1$ and 
\[
E_\lambda=[Sf>\lambda],\quad B_{\lambda,r}=[\cM\charfun_{E_\lambda}>r].
\]
Then we have
\[
E_\lambda\subset B_{\lambda,r}.
\]
\end{lem}
\begin{proof}
Let $t\in E_\lambda$ be fixed. The square function $Sf=\big(\sum_{n=-k+2}^\infty |a_nf_n|^2\big)^{1/2}$ is continuous except possibly at finitely many grid points, where $Sf$ is at least continuous from the right.  As a consequence, for $t\in E_\lambda$, there exists an interval $I\subset E_\lambda$ such that $t\in I$. This implies the following estimate:
\begin{align*}
(\cM\charfun_{E_\lambda})(t)&=\sup_{t\ni U}|U|^{-1}\int_U \charfun_{E_\lambda}(x)\dif x \\
&=\sup_{t\ni U}\frac{|E_\lambda\cap U|}{|U|}\geq \frac{|E_\lambda\cap I|}{|I|}=\frac{|I|}{|I|}=1>r.
\end{align*}
The above inequality shows $t\in B_{\lambda,r}$, proving the lemma.
\end{proof}
\begin{lem}\label{lem:Sf}
Let $f=\sum_{n=-k+2}^\infty a_n f_n$ with only finitely many nonzero coefficients $a_n$, $\lambda>0$ and $r<1$. Then we define
\[
E_\lambda:=[Sf>\lambda],\qquad B_{\lambda,r}:= [\mathcal M\charfun_{E_\lambda}>r].
\]
If 
\[
\Lambda=\{n:J_n\not\subset  B_{\lambda,r}\text{ and }-k+2\leq n<\infty\}\qquad\text{and}\qquad g=\sum_{n\in\Lambda} a_nf_n,
\]
we have
\begin{equation}\label{eq:Sfid}
\int_{E_\lambda}Sg(t)^2\dif t\lesssim_r \int_{E_\lambda^c}Sg(t)^2\dif t.
\end{equation}

\end{lem}
\begin{proof}
First, we observe that in the case $B_{\lambda,r}=[0,1]$, the index set $\Lambda$ is empty and thus, \eqref{eq:Sfid} holds trivially. So let us assume $B_{\lambda,r}\neq [0,1]$.
Then, we  start the proof of \eqref{eq:Sfid} with an application of Lemma \ref{lem:orthsplineJinterval} (for $n\geq 2$) and the fact that $J_n=[0,1]$ for $n\leq 1$ to obtain
\[
\int_{E_\lambda} Sg(t)^2\dif t=\sum_{n\in\Lambda}\int_{E_\lambda}|a_n f_n(t)|^2\dif t\lesssim \sum_{n\in\Lambda} \int_{J_n}|a_nf_n(t)|^2\dif t.
\]
We split the latter expression into the parts
\[
I_1:=\sum_{n\in\Lambda} \int_{J_n\cap E_\lambda^c}|a_nf_n(t)|^2\dif t,\quad I_2:=\sum_{n\in\Lambda} \int_{J_n\cap E_\lambda}|a_nf_n(t)|^2\dif t.
\]
For $I_1$, we clearly have
\begin{equation}\label{eq:lemSf:1}
I_1\leq \sum_{n\in\Lambda} \int_{E_\lambda^c}|a_nf_n(t)|^2\dif t=\int_{E_\lambda^c} Sg(t)^2\dif t.
\end{equation}
It remains to estimate $I_2$. First we observe that by Lemma \ref{lem:triv}, $E_\lambda\subset B_{\lambda,r}$. 
Since the set $B_{\lambda,r}=[\mathcal M\charfun_{E_\lambda}>r]$ is open in $[0,1]$, we decompose it into a countable collection of disjoint open subintervals $(V_j)_{j=1}^\infty$ of $[0,1]$. Utilizing this decomposition, we estimate
\begin{equation}\label{eq:lemSf:2}
I_2\leq \sum_{n\in\Lambda}\sum_{j:|J_n\cap V_j|>0} \int_{J_n\cap V_j}|a_nf_n(t)|^2\dif t.
\end{equation}
If the indices $n$ and $j$ are such that $n\in\Lambda$ and $|J_n\cap V_j|>0$, then, by definition of $\Lambda$, $J_n$ is an interval containing at least one endpoint $x\in\{\inf V_j,\sup V_j\}$ of $V_j$ for which
\[
\cM\charfun_{E_\lambda}(x)\leq r.
\]
This implies
\[
|E_\lambda\cap J_n\cap V_j|\leq r |J_n\cap V_j|\quad\text{or equivalently}\quad |E_\lambda^c\cap J_n\cap V_j|\geq (1-r)|J_n\cap V_j|.
\]
Using this inequality and that $|f_n|^2$ is a polynomial of order $2k-1$ on $J_n$ allows us to use Proposition \ref{prop:poly}  to conclude from \eqref{eq:lemSf:2}
\begin{equation*}
\begin{aligned}
I_2&\lesssim_r \sum_{n\in\Lambda}\sum_{j:|J_n\cap V_j|> 0} \int_{E_\lambda^c\cap J_n\cap V_j}|a_nf_n(t)|^2\dif t \\
&\leq \sum_{n\in\Lambda} \int_{E_\lambda^c\cap J_n\cap B_{\lambda,r}}|a_nf_n(t)|^2\dif t \\
&\leq \sum_{n\in\Lambda} \int_{E_\lambda^c}|a_nf_n(t)|^2\dif t=\int_{E_\lambda^c}Sg(t)^2\dif t,
\end{aligned}
\end{equation*}
The latter inequality combined with \eqref{eq:lemSf:1} completes the proof the lemma.
\end{proof}

\begin{lem}\label{lem:techn2}

Let $V$ be an open subinterval of $[0,1]$, $x:=\inf V$, $y:=\sup V$ and $f=\sum_{n=-k+2}^\infty a_n f_n \in L^p[0,1]$  for $1<p<2$ with $\supp f\subset V$. Let $R>1$ be an arbitrary number satisfying $R\gamma<1$ with the constant $\gamma$ from Theorem \ref{thm:maintool}. Then,
\begin{equation}\label{eq:lemtechn2}
\sum_{n=\polyfun(V)}^\infty R^{p d_n(V)}|a_n|^p \|f_n\|_{L^p(\widetilde{V}^c)}^p\lesssim_{p,R} \|f\|_p^p,
\end{equation}
where $\polyfun(V)=\min\{n:\cT_n\cap V\neq\emptyset\}$ and $\widetilde{V}=(\widetilde{x},\widetilde{y})$ with $\widetilde{x}=x-2|V|$ and $\widetilde{y}=y+2|V|$. 
\end{lem}
\begin{proof}
First observe that $\widetilde{V}^c=[0,\widetilde{x}]\cup [\widetilde{y},1]$. We estimate only the part corresponding to the interval $[0,\widetilde{x}]$ and assume that $\widetilde{x}>0$. The other part is treated analogously.

Let $m\geq 0$ and define
\begin{equation}\label{eq:techn2:0.5}
T_{m}:=\{n\in\bN:n\geq \polyfun(V),\ \card\{i\leq n:\widetilde{x}\leq t_i\leq x\}=m\},
\end{equation}
where $\card E$ is the cardinality of a set $E$. We remark that the index set $T_{m}$ is finite, since the sequence $(t_n)_{n=0}^\infty$ is dense in the unit interval $[0,1]$.

We now split the index set $T_{m}$ further into the following six subcollections.
\begin{align*}
T_{m}^{(1)}&=  \{n\in T_{m}:J_n\subset [\widetilde{x},x]\},\\
T_{m}^{(2)}&=  \{n\in T_{m}:\widetilde{x}\in J_n, |J_n\cap [\widetilde{x},x]|\geq |V|, J_n\not\subset [\widetilde{x},x]\},\\
T_{m}^{(3)}&=  \{n\in T_{m}: J_n\subset [0,\widetilde{x}] \text{ or } \\
&\qquad\big(\widetilde{x}\in J_n \text{ with }|J_n\cap [\widetilde{x},x]|\leq |V| \text{ and }J_n\not\subset [\widetilde{x},x]\big)\},\\
T_{m}^{(4)}&=  \{n\in T_{m}: x\in J_n, |J_n\cap [\widetilde{x},x]|\geq |V|, J_n\not\subset[\widetilde{x},x]\},\\
T_{m}^{(5)}&=  \{n\in T_{m}:J_n\subset [x,\widetilde{y}]\text{ or } \\
&\qquad\big(x\in J_n\text{ with }|J_n\cap[\widetilde{x},x]|\leq |V|\text{ and }J_n\not\subset[\widetilde{x},x]\big)\},\\
T_{m}^{(6)}&=  \{n\in T_{m}:J_n\subset [\widetilde{y},1]\text{ or }\big(\widetilde{y}\in J_n\text{ with }J_n\not\subset [x,\widetilde{y}]\big)\}.
\end{align*}
We treat each of these index sets separately. Before we begin examining sums like in \eqref{eq:lemtechn2} where $n$ is restricted to one of the above index sets, we note that for all $n$ we have by definition of $a_n=\langle f,f_n\rangle$ and the support assumption on $f$
\begin{equation}\label{eq:techn2:1}
|a_n|^p\leq \int_V |f(t)|^p\dif t\cdot \Big(\int_V |f_n(t)|^{p'}\dif t\Big)^{p-1},
\end{equation}
where $p'=p/(p-1)$ denotes the conjugate Hölder exponent to $p$.

\textsc{Case 1: } $n\in T_{m}^{(1)}=\{n\in T_{m}:J_n\subset [\widetilde{x},x]\}$.\\ Let $\widetilde{T}_{m}^{(1)}:=T_{m}^{(1)}\setminus \{\min T_{m}^{(1)}\}$. By definition, the interval $J_n$ is at most $k-1$ grid points in $\cT_n$ away from $t_n$. Since the number $m$ of grid points between $\widetilde{x}$ and $x$ is constant for all $n\in T_{m}$, there are only $2(k-1)$ possibilities for $J_n$ with $n\in \widetilde{T}_{m}^{(1)}$. By Lemma \ref{lem:jinterval}, applied with $\beta=0$, every $J_n$ is characteristic interval of at most $F_{k}$ points $t_m$ and thus,
\begin{equation}\label{eq:techn2:2}
\card T_{m}^{(1)}\leq 2(k-1)F_k+1.
\end{equation}
By Lemma \ref{lem:lporthspline} and Lemma \ref{lem:orthsplineJinterval} respectively, 
\begin{equation}\label{eq:techn2:3}
\int_0^{\widetilde{x}}|f_n(t)|^p\dif t\lesssim \gamma^{p d_n(\widetilde{x})}\|f_n\|_p^p\quad\text{ and } \int_V |f_n(t)|^{p'}\dif t\lesssim \gamma^{p' d_n(V)}\|f_n\|_{p'}^{p'}
\end{equation}
for $n\in T_{m}^{(1)}$.
Furthermore, $d_n(\widetilde{x})+d_n(V)=m$ by definition of $d_n$, the location of $J_n$ and the fact that $n\in T_{m}^{(1)}$. So, using \eqref{eq:techn2:1}, \eqref{eq:techn2:3} and Lemma \ref{lem:orthsplineJinterval} respectively,
\begin{align*}
\sum_{n\in T_{m}^{(1)}}& R^{p d_n(V)} |a_n|^p \int_0^{\widetilde{x}} |f_n(t)|^p\dif t   \\
&\leq \sum_{n\in T_{m}^{(1)}} R^{p d_n(V)} \int_V |f(t)|^p\dif t\cdot\Big(\int_V|f_n(t)|^{p'}\dif t\Big)^{p-1} \int_0^{\widetilde{x}}|f_n(t)|^p\dif t\\
&\lesssim \sum_{n\in T_{m}^{(1)}} R^{pd_n(V)} \gamma^{p(d_n(\widetilde{x})+d_n(V))} \|f_n\|_p^p \|f_n\|_{p'}^p \int_V|f(t)|^p\dif t \\
&\lesssim \sum_{n\in T_{m}^{(1)}} (R\gamma)^{pm} \int_{V} |f(t)|^p\dif t.
\end{align*}
Finally, we employ \eqref{eq:techn2:2} to obtain
\begin{equation}\label{eq:techn2:4}
\sum_{n\in T_{m}^{(1)}} R^{p d_n(V)} |a_n|^p \int_0^{\widetilde{x}} |f_n(t)|^p\dif t \lesssim (R\gamma)^{pm}\int_V |f(t)|^p\dif t,
\end{equation}
which concludes the proof of Case 1.

\textsc{Case 2: } $n\in T_{m}^{(2)}=\{n\in T_{m}:\widetilde{x}\in J_n, |J_n\cap [\widetilde{x},x]|\geq |V|, J_n\not\subset [\widetilde{x},x]\}$. \\
In this case we have $d_n(V)=m$ and thus Lemma \ref{lem:lporthspline} implies
\begin{equation*}
\int_V |f_n(t)|^{p'}\dif t\leq \|f_n\|_{L^\infty(V)}^{p'}|V| \lesssim \gamma^{p'm}|J_n|^{-p'/2}|V|.
\end{equation*}
So we use \eqref{eq:techn2:1} and this estimate and  to obtain
\begin{equation*}
\begin{aligned}
|a_n|^p\|f_n\|_p^p &\leq \int_V |f(t)|^p\dif t\cdot \Big(\int_V |f_n(t)|^{p'}\dif t\Big)^{p-1} \|f_n\|_p^p \\
&\lesssim \int_V |f(t)|^p\dif t\cdot \gamma^{pm}|J_n|^{-p/2}|V|^{p-1} \|f_n\|_p^p.
\end{aligned}
\end{equation*}
We continue and employ Lemma \ref{lem:orthsplineJinterval} to get further
\begin{equation}\label{eq:techn2:5}
\begin{aligned}
|a_n|^p\|f_n\|_p^p &\lesssim \gamma^{pm} |J_n|^{-p/2+1-p/2}|V|^{p-1} \int_V |f(t)|^p\dif t \\
&\leq \gamma^{pm}|J_n|^{1-p}|V|^{p-1}\|f\|_p^p.
\end{aligned}
\end{equation}
If $n_0<n_1<\dots <n_s$ is an enumeration of all elements in $T_{m}^{(2)}$, we have by definition of $T_{m}^{(2)}$
\begin{equation*}
J_{n_0}\supset J_{n_1}\supset \dots \supset J_{n_s}\qquad \text{and} \qquad|J_{n_s}|\geq |V|.
\end{equation*}
Thus, Lemma \ref{lem:jinterval} and the fact that $1<p<2$ imply 
\begin{equation}\label{eq:techn2:6}
\sum_{n\in T_{m}^{(2)}} |J_n|^{1-p} \sim_p |J_{n_s}|^{1-p}\leq |V|^{1-p}
\end{equation}
We finally use \eqref{eq:techn2:5} and \eqref{eq:techn2:6} to conclude
\begin{equation}\label{eq:techn2:7}
\begin{aligned}
\sum_{n\in T_{m}^{(2)}}  R^{p d_n(V)} |a_n|^p \|f_n\|_p^p&\lesssim (R\gamma)^{pm}|V|^{p-1}\|f\|_p^p \sum_{n\in T_{m}^{(2)}} |J_n|^{1-p} \\
&\lesssim_p (R\gamma)^{pm} \|f\|_p^p.
\end{aligned}
\end{equation}

\textsc{Case 3: }$n\in T_{m}^{(3)}=\{n\in T_{m}: J_n\subset [0,\widetilde{x}] \text{ or } \big(\widetilde{x}\in J_n \text{ with }|J_n\cap [\widetilde{x},x]|\leq |V| \text{ and }J_n\not\subset [\widetilde{x},x]\big)\}$. 

For $n\in T_{m}^{(3)}$, we denote by the finite sequence $(x_i)_{i=1}^m$ the points in $\cT_n\cap [\widetilde{x},x]$ in increasing order and counting multiplicities. 
If there exists an index $n\in T_{m}^{(3)}$ such that $x_1$ is the right endpoint of $J_n$ and $\widetilde{x}\in J_n$, we define $x^*:=x_1$. If not, we set $x^*:=\widetilde{x}$. By definition of $T_{m}^{(3)}$ and $x^*$, we have
\begin{equation}\label{eq:techn2:8}
|V|\leq |[x^*,x]|\leq 2|V|.
\end{equation}
Furthermore, for all $n\in T_{m}^{(3)}$,
\begin{equation*}
J_n\subset [0,x^*]\quad \text{and}\quad |[x^*,x]\cap \cT_n|=m.
\end{equation*}
Moreover,
\begin{equation}\label{eq:techn2:10}
m+d_n(x^*)-k\leq d_n(V)\leq m+d_n(x^*),
\end{equation}
where the exact value of $d_n(V)$ depends on the multiplicity of $x^*$ in $\cT_n$ (which cannot exceed $k$). By Lemma \ref{lem:lporthspline} and \eqref{eq:techn2:10} we have
\begin{equation*}
\sup_{t\in V} |f_n(t)|\lesssim \gamma^{m+d_n(x^*)}\frac{|J_n|^{1/2}}{|J_n|+\dist(x,J_n)}.
\end{equation*}
We use this inequality to get
\begin{equation}\label{eq:techn2:12}
\int_V |f_n(t)|^{p'}\dif t\lesssim |V|\cdot \gamma^{p'(m+d_n(x^*))}\frac{|J_n|^{p'/2}}{(|J_n|+\dist(x,J_n))^{p'}}.
\end{equation}
Employing \eqref{eq:techn2:1}, \eqref{eq:techn2:12} and Lemma \ref{lem:orthsplineJinterval} respectively,
\begin{equation*}
\begin{aligned}
 R^{p d_n(V)}&|a_n|^p \|f_n\|_p^p  \\
&\leq R^{pd_n(V)}\int_V |f(t)|^p\dif t\cdot\Big(\int_V |f_n(t)|^{p'}\dif t\Big)^{p-1}\|f_n\|_p^p \\
&\lesssim R^{pd_n(V)}\|f\|_p^p |V|^{p-1}\gamma^{p(m+d_n(x^*))}\frac{|J_n|^{p/2}}{(|J_n|+\dist(x,J_n))^{p}} \|f_n\|_p^p \\
&\lesssim R^{pd_n(V)}\|f\|_p^p |V|^{p-1}\gamma^{p(m+d_n(x^*))}\frac{|J_n|}{(|J_n|+\dist(x,J_n))^{p}}. 
\end{aligned}
\end{equation*}
Inequality \eqref{eq:techn2:10} then yields
\begin{equation}\label{eq:techn2:13}
 R^{p d_n(V)}|a_n|^p \|f_n\|_p^p  \leq (R\gamma)^{p(m+d_n(x^*))}\|f\|_p^p |V|^{p-1}\frac{|J_n|}{(|J_n|+\dist(x,J_n))^{p}}.
\end{equation}
We now have to sum this inequality. In order to do this we split our analysis depending on the value of $d_n(x^*)$. For fixed $j\in\bN_0$ we view $n\in T_{m}^{(3)}$ with $d_n(x^*)=j$. Let $\beta=1/4$, then, by Lemma \ref{lem:jinterval}, each point $t$ (which is not a grid point) belongs to at most $F_{k}$ intervals $J_{n}^\beta$ with $n\in T_{m}^{(3)}$ and $d_n(x^*)=j$. Here $J_n^\beta$ is the unique closed interval that satisfies the requirements
\[
|J_n^\beta|=\beta|J_n|\quad\text{and}\quad \inf J_n^\beta=\inf J_n.
\]
Furthermore, for $t\in J_n$, we have
\begin{equation*}
|J_n|+\dist(x,J_n)\geq x-t.
\end{equation*}
These facts allow us to estimate
\begin{equation*}
\begin{aligned}
\sum_{\substack{n\in T_{m}^{(3)}\\ d_n(x^*)=j}}\frac{|J_n||V|^{p-1}}{(|J_n|+\dist(x,J_n))^p}&\leq \beta^{-1} \sum_{\substack{n\in T_{m}^{(3)}\\ d_n(x^*)=j}}\int_{J_n^\beta}\frac{|V|^{p-1}}{(x-t)^p}\dif t \\
&\leq \frac{F_{k}}{\beta} |V|^{p-1}\int_{-\infty}^{x^*}(x-t)^{-p} \dif t\\
&\lesssim_{p} \frac{|V|^{p-1}}{(x-x^*)^{p-1}}\leq 1,
\end{aligned}
\end{equation*}
where in the last step we used \eqref{eq:techn2:8}.
Combining \eqref{eq:techn2:13} and the latter and summing over $j$ (here we use the fact that $R\gamma<1$), we arrive at
\begin{equation}\label{eq:techn2:16}
\sum_{n\in T_{m}^{(3)}} R^{pd_n(V)}|a_n|^p\|f_n\|_p^p\lesssim_{p,R} (R\gamma)^{pm}\|f\|_p^p.
\end{equation}

\textsc{Case 4: }$n\in T_{m}^{(4)}=\{n\in T_{m}: x\in J_n, |J_n\cap [\widetilde{x},x]|\geq |V|, J_n\not\subset[\widetilde{x},x]\}$. \\
We can ignore the cases $(m=0)$ or ($m=1$ and $[\widetilde{x},x]\cap \cT_n=\{x\}$) since these are settled in Case 2. We thus define $\widetilde{T}_{m}^{(4)}$ as the set of all remaining indices from $T_{m}^{(4)}$. Let $n\in \widetilde{T}_{m}^{(4)}$. Then the definition of $T_{m}^{(4)}$ implies 
\begin{equation}\label{eq:techn2:17}
d_n(V)=d_n([x,y])=0.
\end{equation}
Moreover, there exists at least one point of $\cT_n$ in $V$ (since $n\geq \polyfun(V)$ for $n\in T_m$) and at least one point of $\cT_n$ in $[\widetilde{x},x]$ (since $m\geq 1$). Thus we have the following two-sided bound on $|J_n|$:
\begin{equation}\label{eq:techn2:18}
|V|\leq |J_n|\leq 3|V|.
\end{equation}
Since $x\in J_n$ for all $n\in \widetilde{T}_{m}^{(4)}$, the family $\{J_n:n\in \widetilde{T}_m^{(4)}\}$ forms a decreasing collection of sets. Inequality \eqref{eq:techn2:18} and a multiple application of Lemma \ref{lem:jinterval} with sufficiently large $\beta$ gives us a constant $c_k$ depending only on $k$ such that
\begin{equation}\label{eq:techn2:19}
\card \widetilde{T}_{m}^{(4)}\leq c_k.
\end{equation}
We employ Lemma \ref{lem:lporthspline} and Lemma \ref{lem:orthsplineJinterval} respectively to get
\begin{equation}\label{eq:techn2:20}
\int_0^{\widetilde{x}}|f_n(t)|^p\dif t\lesssim \gamma^{pm}|J|^{p/2-p+1}=\gamma^{pm}|J|^{1-p/2}\lesssim \gamma^{pm}\|f_n\|_p^p.
\end{equation}
Thus we are able to conclude
\begin{equation*}
\begin{aligned}
\sum_{n\in \widetilde{T}_{m}^{(4)}}& R^{p d_n(V)} |a_n|^p\int_0^{\widetilde{x}}|f_n(t)|^p\dif t \\
&\lesssim \sum_{n\in \widetilde{T}_{m}^{(4)}}\int_V |f(t)|^p\dif t\cdot \Big(\int_V |f_n(t)|^{p'}\dif t\Big)^{p-1} \int_0^{\widetilde{x}}|f_n(t)|^p\dif t \\
&\lesssim
\sum_{n\in \widetilde{T}_{m}^{(4)}}\int_V |f(t)|^p\dif t\cdot \|f_n\|_{p'}^p \gamma^{pm}\|f_n\|_p^p, \\
&\leq\sum_{n\in \widetilde{T}_m^{(4)}}\gamma^{pm}\|f\|_p^p,
\end{aligned}
\end{equation*}
where we used \eqref{eq:techn2:17} and \eqref{eq:techn2:1} in the first inequality, \eqref{eq:techn2:20} in the second inequality and Lemma \ref{lem:orthsplineJinterval} in the last inequality.
Consequently, considering \eqref{eq:techn2:19}, the latter display implies
\begin{equation}\label{eq:techn2:22}
\sum_{n\in \widetilde{T}_{m}^{(4)}}R^{p d_n(V)}|a_n|^p\int_0^{\widetilde{x}} |f_n(t)|^p\dif t\lesssim \gamma^{pm}\|f\|_p^p.
\end{equation}

\textsc{Case 5: }$n\in T_{m}^{(5)}=\{n\in T_{m}:J_n\subset [x,\widetilde{y}]\text{ or }\big(x\in J_n\text{ with }|J_n\cap[\widetilde{x},x]|\leq |V|\text{ and }J_n\not\subset[\widetilde{x},x]\big)\}.$\\
If there exists $n\in T_{m}^{(5)}$ with $x_m=\inf J_n$, then we define $x'=x_m$. If there exists no such index, we set $x'=x$. 
We now fix $n\in T_m^{(5)}$. By definition of $x'$ and $\widetilde{x}$,
\begin{equation}\label{eq:techn2:25}
m+d_n(x')-k\leq d_n(\widetilde{x})\leq m+d_n(x').
\end{equation}
The exact relation between $d_n(\widetilde{x})$ and $d_n(x')$ depends on the multiplicity of the point $x'$ in the grid $\cT_n$. By definition of $T_m^{(5)}$,
\[
d(\widetilde{x},J_n)\leq 5|V|\qquad\text{and}\qquad |V|\leq d(\widetilde{x},J_n).
\]
Moreover,
\begin{equation}\label{eq:techn2:26}
|J_n|\leq |[x',\widetilde{y}]|\leq 4|V|\qquad\text{and}\qquad d_n(V)\leq d_n(x').
\end{equation}
The latter two displays now imply
\begin{equation*}
|J_n|+\dist(\widetilde{x},J_n)\sim |V|.
\end{equation*}
Lemma \ref{lem:lporthspline}, together with the former observation, yields
\begin{align*}
\int_0^{\widetilde{x}}|f_n(t)|^p\dif t&\lesssim  \gamma^{p d_n(\widetilde{x})}\frac{|J_n|^{p/2}}{(|J_n|+\dist(\widetilde{x},J_n))^{p-1}} \\
&\lesssim \gamma^{pd_n(\widetilde{x})}\frac{|J_n|^{p/2}}{|V|^{p-1}}.
\end{align*}
Inserting \eqref{eq:techn2:25} in this inequality, we get
\begin{equation}\label{eq:techn2:29}
\int_0^{\widetilde{x}}|f_n(t)|^p\dif t\lesssim\gamma ^{p(d_n(x')+m)}\frac{|J_n|^{p/2}}{|V|^{p-1}}.
\end{equation}

For each $n\in T_{m}^{(5)}$, we split the interval $[x',\widetilde{y}]$ into the union of three disjoint subintervals $I_\ell$, $1\leq \ell\leq 3$, defined by 
\begin{equation*}
I_1:=[x',\inf J_n],\qquad I_2:=J_n, \qquad I_3:=[\sup J_n,\widetilde{y}].
\end{equation*}
Corresponding to these subintervals, we set
\begin{equation*}
a_{n,\ell}:=\int_{I_\ell\cap V} f(t)f_n(t)\dif t,\qquad\ell=1,2,3.
\end{equation*}

We start with analyzing the parameter choice $\ell=2$ and first observe that by definition of $I_2$, 
\begin{equation}\label{eq:techn2:32}
|a_{n,2}|^p\leq \|f_n\|_{p'}^p \int_{J_n}|f(t)|^p\dif t.
\end{equation}
We split the index set $T_{m}^{(5)}$ further and look at the set of those $n\in T_m^{(5)}$ such that $d_n(x')=j$ for fixed $j\in\mathbb{N}_0$. These indices $n$ may be arranged in packets such that the intervals $J_n$ from one packet have the same left endpoint and such that the maximal intervals of different packets are disjoint. Observe that the intervals $J_n$ from one packet form a decreasing collection of sets. Let $J_{n_0}$ be the maximal interval of one packet. Define the index set $\mathcal{I}_j:=\{n\in T_m^{(5)}: d_n(x')=j,\ J_n\subset J_{n_0}\}$. Then we use \eqref{eq:techn2:26} and \eqref{eq:techn2:32} to estimate
\begin{align*}
E_{2,j}&:=\sum_{n\in\mathcal{I}_j} R^{p d_n(V)}|a_{n,2}|^p\int_{0}^{\widetilde{x}}|f_n(t)|^p\dif t \\
&\leq \sum_{n\in \mathcal{I}_j} R^{pj} \|f_n\|_{p'}^p \int_{J_n}|f(t)|^p\dif t \int_0^{\widetilde{x}} |f_n(t)|^p\dif t.
\end{align*}
We continue and use \eqref{eq:techn2:29} to get 
\[
E_{2,j}\lesssim  R^{pj}\int_{J_{n_0}}|f(t)|^p\dif t \sum_{n\in\mathcal I_j} \|f_n\|_{p'}^p \gamma^{p(d_n(x')+m)}\frac{|J_n|^{p/2}}{|V|^{p-1}}.
\]
By Lemma \ref{lem:orthsplineJinterval}, $\|f_n\|_{p'}\sim |J|^{1/p'-1/2}$, an thus,
\[
E_{2,j}\lesssim (R\gamma)^{pj} \gamma^{pm} \int_{J_{n_0}} |f(t)|^p\dif t \cdot\sum_{n\in\mathcal{I}_j} \frac{|J_n|^{p-1}}{|V|^{p-1}}.
\]
We apply Lemma \ref{lem:jinterval} to the above sum and conclude
\begin{align*}
E_{2,j}&\lesssim_p (R\gamma)^{pj}\gamma^{pm}\int_{J_{n_0}}|f(t)|^p\dif t\cdot \frac{|J_{n_0}|^{p-1}}{|V|^{p-1}} \\
&\lesssim (R\gamma)^{pj}\gamma^{pm}\int_{J_{n_0}}|f(t)|^p\dif t,
\end{align*}
where in the last inequality, we used \eqref{eq:techn2:26}.
Now, summing over all maximal intervals $J_{n_0}$ and over $j$ finally yields (note that $R\gamma<1$) 
\begin{equation}\label{eq:techn2:34}
\sum_{n\in T_{m}^{(5)}} R^{p d_n(V)}|a_{n,2}|^p \int_{0}^{\widetilde{x}} |f_n(t)|^p\dif t \lesssim_{p,R} \gamma^{pm}\|f\|_p^p.
\end{equation}
This completes the proof of the part $\ell=2$.

We continue with the parameter choice $\ell=3$. Let $j\in \bN_0$ fixed and let $(n_{j,r})_{r=1}^\infty$ be the subsequence of all $n\in T_{m}^{(5)}$ with $d_n(x')=j$. For two such indices $n_1<n_2$ we have either
\begin{equation*}
(\inf J_{n_1}=\inf J_{n_2}\text{ and }J_{n_2}\subset J_{n_1})\quad\text{or}\quad \sup J_{n_2}\leq \inf J_{n_1}.
\end{equation*}
Observe that $J_{n_2}=J_{n_1}$ is possible, but by Lemma \ref{lem:jinterval} (with $\beta=0$) only $F_k$ times with $F_k$ only depending on $k$. Therefore, with $\beta_{n_{j,r}}:=\sup J_{n_{j,r}}$ for $r\geq 1$ and $\beta_{n_{j,0}}:=\widetilde{y}$,
\begin{equation*}
d_{n_{j,s}}(\beta_{n_{j,r}})\geq \frac{s-r}{F_k}-1,\qquad s\geq r\geq 1.
\end{equation*}
Thus we obtain for $s\geq r\geq 1$ by Lemma \ref{lem:lporthspline} and Lemma \ref{lem:orthsplineJinterval}
\begin{equation}\label{eq:techn2:37}
\begin{aligned}
\int_{\beta_{n_{j,r}}}^{\beta_{n_{j,r-1}}} |f_{n_{j,s}}(t)|^{p'}\dif t&\lesssim \gamma^{p'd_{n_{j,s}}(\beta_{n_{j,r}})}  \|f_{n_{j,s}}\|_{p'}^{p'} 
\lesssim \gamma^{p'\frac{s-r}{F_k}} \|f_{n_{j,s}}\|_{p'}^{p'}
\end{aligned}
\end{equation}
and similarly, using also \eqref{eq:techn2:25},
\begin{equation}\label{eq:techn2:37.5}
\int_{0}^{\widetilde{x}}|f_{n_{j,s}}|^p\dif t\lesssim \gamma^{p d_{n_{j,s}}(\widetilde{x})}\|f_{n_{j,s}}\|^p_p \lesssim \gamma^{p (m+d_{n_{j,s}}(x'))}\|f_{n_{j,s}}\|^p_p.
\end{equation}
Choosing $\kappa:=\gamma^{1/(2F_k)}<1$, we conclude
\begin{equation*}
\begin{aligned}
|a_{n_{j,s},3}|^p&=\Big|\int_{\beta_{n_{j,s}}}^{\widetilde{y}} f(t) f_{n_{j,s}}(t)\dif t \Big|^p \\
&=\Big|\sum_{r=1}^s \kappa^{s-r}\kappa^{r-s}\int_{\beta_{n_{j,r}}}^{\beta_{n_{j,r-1}}} f(t) f_{n_{j,s}}(t)\dif t\Big|^p \\
&\leq \Big(\sum_{r=1}^s\kappa^{p'(s-r)}\Big)^{p/p'}\sum_{r=1}^s \kappa^{p(r-s)}\Big|\int_{\beta_{n_{j,r}}}^{\beta_{n_{j,r-1}}} f(t) f_{n_{j,s}}(t)\dif t\Big|^p \\
&\lesssim \sum_{r=1}^s \kappa^{p(r-s)} \int_{\beta_{n_{j,r}}}^{\beta_{n_{j,r-1}}} |f(t)|^p \dif t\cdot
\Big(\int_{\beta_{n_{j,r}}}^{\beta_{n_{j,r-1}}} |f_{n_{j,s}}(t)|^{p'}\dif t\Big)^{p/p'}. \\
\end{aligned}
\end{equation*}
We use inequality \eqref{eq:techn2:37} to obtain from the latter expression
\begin{equation}\label{eq:techn2:38}
|a_{n_{j,s},3}|^p \lesssim\sum_{r=1}^s  \gamma^{p\frac{s-r}{2F_k}} \int_{\beta_{n_{j,r}}}^{\beta_{n_{j,r-1}}} |f(t)|^p \dif t \cdot \|f_{n_{j,s}}\|_{p'}^p.
\end{equation}
Combining \eqref{eq:techn2:38} and \eqref{eq:techn2:37.5} yields
\begin{equation*}
\begin{aligned}
E_{3,j}&:=\sum_{\substack{n\in T_{m}^{(5)}\\d_n(x')=j}} R^{p d_n(V)}|a_{n,3}|^p \|f\|_{L^p(0,\widetilde{x})}^p  \\
&= \sum_{s\geq 1} R^{pj}|a_{n_{j,s},3}|^p \|f_{n_{j,s}}\|_{L^p(0,\widetilde{x})}^p \\
&\lesssim \sum_{s\geq 1}R^{pj}\sum_{r=1}^s \gamma^{p\frac{s-r}{2F_k}}\|f_{n_{j,s}}\|_{p'}^p \int_{\beta_{n_{j,r}}}^{\beta_{n_{j,r-1}}} |f(t)|^p \dif t\cdot \gamma^{p(m+j)} \|f_{n_{j,s}}\|^p_p. \\
\end{aligned}
\end{equation*}
Using again Lemma \ref{lem:orthsplineJinterval} gives
\begin{align*}
E_{3,j}&\lesssim \gamma^{pm}(R\gamma)^{pj}\sum_{r\geq 1} \int_{\beta_{n_{j,r}}}^{\beta_{n_{j,r-1}}} |f(t)|^p \dif t\sum_{s\geq r}\gamma^{p\frac{s-r}{2F_k}} \\
&\lesssim \gamma^{pm}(R\gamma)^{pj}\|f\|_p^p.
\end{align*}
Summing over $j$ finally yields
\begin{equation}\label{eq:techn2:40}
\sum_{n\in T_{m}^{(5)}} R^{p d_n(V)}|a_{n,3}|^p\|f\|_{L^p(0,\widetilde{x})}^p\lesssim_{p,R} \gamma^{pm}\|f\|_p^p,
\end{equation}
since $R\gamma<1$. This finishes the proof of the part $\ell=3$.

We now come to the final part $\ell=1$. Let $j$ and $n$ be fixed such that $d_n(x')=j$ and let $L_{1,n},\dots, L_{j,n}$ be the grid intervals in the grid $\mathcal T_n$ between $x'$ and $J_n$ from left to right. Observe that $f_n$ is a polynomial on each of the intervals  $L_{i,n}$. We define
\begin{equation*}
b_{i,n}:=\int_{L_{i,n}} f(t)f_n(t)\dif t,\qquad 1\leq i\leq j.
\end{equation*}
For $n$ with $d_n(x')=j$, we clearly have $a_{n,1}=\sum_{i=1}^j b_{i,n}$ and Hölder's inequality implies
\begin{equation}\label{eq:techn2:42}
|b_{i,n}|^p\leq \int_{L_{i,n}} |f(t)|^p\dif t \cdot\Big(\int_{L_{i,n}}|f_n(t)|^{p'}\dif t\Big)^{p/p'}.
\end{equation}
Remark \ref{rem:linftyorthspline} yields the bound
\begin{equation*}
\sup_{t\in L_{i,n}}|f_n(t)|\lesssim \gamma^{j-i}\frac{|J_n|^{1/2}}{|J_n|+\dist(J_n,L_{i,n})+|L_{i,n}|}
\end{equation*}
and inserting this in \eqref{eq:techn2:42} gives
\begin{equation}\label{eq:techn2:44}
|b_{i,n}|^p\leq \int_{L_{i,n}} |f(t)|^p\dif t\cdot \gamma^{p(j-i)}\frac{|J_n|^{p/2}|L_{i,n}|^{p-1}}{(|J_n|+\dist(J_n,L_{i,n})+|L_{i,n}|)^p}.
\end{equation}
Observe that we have the elementary inequality
\begin{equation}\label{eq:techn2:45}
\begin{aligned}
\frac{|J_n|^{p/2}|L_{i,n}|^{p-1}}{(|J_n| +\dist(J_n,L_{i,n})+|L_{i,n}|)^p}&\frac{|J_n|^{p/2}}{|V|^{p-1}}\leq 
\\ & \frac{|J_n|}{|V|^{p-1}}(|J_n|+ \dist(J_n,L_{i,n})+|L_{i,n}|)^{p-2}.
\end{aligned}
\end{equation}
Combining \eqref{eq:techn2:44}, \eqref{eq:techn2:45} and \eqref{eq:techn2:29} allows us to estimate (recall that we assumed $n$ is such that $d_n(x')=j$)
\begin{equation}\label{eq:techn2:46}
\begin{aligned}
R&^{pd_n(V)}|b_{i,n}|^p \cdot \int_0^{\widetilde{x}} |f_n(t)|^p\dif t\\
&\lesssim R^{pj} \gamma^{p(j-i)}\int_{L_{i,n}} |f(t)|^p\dif t \frac{|J_n|^{p/2}|L_{i,n}|^{p-1}}{(|J_n|+\dist(J_n,L_{i,n})+|L_{i,n}|)^p}\cdot \gamma^{p(j+m)}\frac{|J_n|^{p/2}}{|V|^{p-1}} \\
&\lesssim R^{pj}\gamma^{p(2j+m-i)} \frac{|J_n|}{|V|^{p-1}}(|J_n|+\dist(J_n,L_{i,n})+|L_{i,n}|)^{p-2}\int_{L_{i,n}}|f(t)|^p\dif t.
\end{aligned}
\end{equation}
For fixed $j$ and $i$ we view those indices $n$ such that $d_n(x')=j$ and consider the corresponding intervals $L_{i,n}$. These intervals can be collected in packets such that intervals $L_{i,n}$ from one packet have the same left endpoint and maximal intervals of different packets are disjoint. 
For $\beta=1/4$, we denote by $J_{n}^\beta$ the unique interval that has the same right endpoint as $J_n$ and length $\beta |J_n|$. The intervals $J_n$ corresponding to $L_{i,n}$'s from one packet can now be grouped in the same way as the $L_{i,n}$'s and thus, Lemma \ref{lem:jinterval} implies the existence of a constant $F_{k}$ depending only on $k$ such that every point $t\in[0,1]$ belongs to at most $F_k$ intervals $J_n^\beta$ corresponding to the intervals $L_{i,n}$ from one packet.
We now consider one such packet and denote by $u^*$ the left endpoint of (all) intervals $L_{i,n}$ in this packet. Then we have for $t\in J_n^\beta$ 
\begin{equation}\label{eq:techn2:47}
|J_n|+\dist(L_{i,n},J_n)+|L_{i,n}|\geq |t-u^*|.
\end{equation}
If $L_{i}^*$ is the maximal interval of the present packet, \eqref{eq:techn2:46} and \eqref{eq:techn2:47} yield 
\begin{equation*}
\begin{aligned}
&\hspace{-1cm}\sum_{n:L_{i,n}\text{ in one packet}} R^{p d_n(V)}|b_{i,n}|^p \|f_n\|_{L^p(0,\widetilde{x})}^p \\
&\lesssim \frac{R^{pj}\gamma^{p(2j+m-i)}}{|V|^{p-1}} \sum_n|J_n|(|J_n|+\dist(L_{i,n},J_n)+|L_{i,n}|)^{p-2}\int_{L_{i,n}}|f(t)|^p\dif t \\
&\lesssim \frac{R^{pj}\gamma^{p(2j+m-i)}}{|V|^{p-1}}\int_{L_{i}^*}|f(t)|^p\dif t\cdot \sum_n \int_{J_n^\beta} |t-u^*|^{p-2}\dif t.
\end{aligned}
\end{equation*}
Since every point $t$ belongs to at most $F_k$ intervals $J_n^\beta$ in one package of $L_{i,n}$'s, we can continue this chain of inequalities and get further, by using the facts $J_n\subset [x',\widetilde{y}]$ and $p<2$:
\begin{equation*}
\begin{aligned}
&\hspace{-1.5cm}\sum_{n:L_{i,n}\text{ in one packet}} R^{p d_n(V)}|b_{i,n}|^p \|f_n\|_{L^p(0,\widetilde{x})}^p \\
&\lesssim \frac{R^{pj}\gamma^{p(2j+m-i)}}{|V|^{p-1}}\int_{L_{i}^*}|f(t)|^p\dif t\cdot \int_{u^*}^{\widetilde{y}} |t-u^*|^{p-2}\dif t \\
&\lesssim R^{pj}\gamma^{p(2j+m-i)}\int_{L_{i}^*}|f(t)|^p\dif t,
\end{aligned}
\end{equation*}
where in the last inequality we used \eqref{eq:techn2:26}. Since the maximal intervals $L_i^*$ of different packets are disjoint, we can sum over all packets (for fixed $j$ and $i$) to obtain 
\begin{equation}\label{eq:techn2:50}
\sum_{\substack{n\in T_{m}^{(5)}\\d_n(x')=j}} R^{p d_n(V)}|b_{i,n}|^p\|f_n\|_{L^p(0,\widetilde{x})}^p \lesssim    R^{pj}\gamma^{p(2j+m-i)}\|f\|_p^p.
\end{equation}
Let $\kappa:=\gamma^{1/2}<1$. Then, for $n$ such that $d_n(x')=j$ we have 
\begin{equation}\label{eq:techn2:51}
|a_{n,1}|^p=\Big|\sum_{i=1}^j b_{i,n}\Big|^p=\Big| \sum_{i=1}^j \kappa^{j-i}\kappa^{i-j}b_{i,n}\Big|^p \lesssim_p\sum_{i=1}^j \kappa^{p(i-j)}|b_{i,n}|^p
\end{equation}
Combining \eqref{eq:techn2:51} with \eqref{eq:techn2:50} we get 
\begin{equation*}
\begin{aligned}
\sum_{\substack{n\in T_{m}^{(5)}\\d_n(x')=j}}& R^{p d_n(V)}|a_{1,n}|^p\|f_n\|_{L^p(0,\widetilde{x})}^p \\ &\lesssim_p \sum_{i=1}^j \kappa^{p(i-j)}\sum_{\substack{n\in T_{m}^{(5)}\\d_n(x')=j}} R^{p d_n(V)}|b_{i,n}|^p\|f_n\|_{L^p(0,\widetilde{x})}^p \\
&\lesssim \sum_{i=1}^j \kappa^{p(i-j)} R^{pj}\gamma^{p(2j+m-i)}\|f\|_p^p\lesssim (R\gamma)^{pj}\gamma^{pm}\|f\|_p^p.
\end{aligned}
\end{equation*}
Since $R\gamma<1$ we sum over $j$ to conclude finally 
\begin{equation}\label{eq:techn2:53}
\sum_{n\in T_{m}^{(5)}} R^{p d_n(V)}|a_{n,1}|^p \|f_n\|_{L^p(0,\widetilde{x})}^p \lesssim_{p,R} \gamma^{pm}\|f\|_p^p
\end{equation}
This finishes the proof of case $\ell=1$.

We can now combine the proved inequalities for $\ell=1,2,3$, that is \eqref{eq:techn2:53}, \eqref{eq:techn2:34} and \eqref{eq:techn2:40}, to complete the analysis of Case 5 with the estimate
\begin{equation}\label{eq:techn2:54}
\sum_{n\in T_{m}^{(5)}} R^{pd_n(V)} |a_n|^p \|f_n\|_{L^p(0,\widetilde{x})}^p\lesssim_{p,R} \gamma^{pm}\|f\|_p^p.
\end{equation}

\textsc{Case 6: }$n\in T_{m}^{(6)}=\{n\in T_{m}:J_n\subset [\widetilde{y},1]\text{ or }\big(\widetilde{y}\in J_n\text{ with }J_n\not\subset [x,\widetilde{y}]\big)\}$. \\
Similarly to \eqref{eq:techn2:0.5}, we may use the symmetric splitting of the indices $n$ to
\begin{equation*}
T_{\operatorname{r},s}:=\{n\geq \polyfun(V):|[y,\widetilde{y}]\cap \cT_n|=s\},
\end{equation*}
where $\operatorname{r}$ stands for ``right". These collections of indices are again splitted into six subcollections $T_{\operatorname{r},s}^{(i)}$, $1\leq i\leq 6$, where the two of interest are
\begin{equation*}\label{eq:techn2:56}
\begin{aligned}
T_{\operatorname{r},s}^{(2)}&=\{n\in T_{r,s}:\widetilde{y}\in J_n, |J_n\cap [y,\widetilde{y}]|\geq |V|, J_n\not\subset [y,\widetilde{y}]\}, \\
T_{\operatorname{r},s}^{(3)}&=\{ n\in T_{r,s}: J_n\subset [\widetilde{y},1] \text{ or } \\ &\qquad\big(\widetilde{y}\in J_n \text{ with } |J_n\cap [y,\widetilde{y}]|\leq |V| \text{ and } J_n\not\subset[y,\widetilde{y}]\big)\}.
\end{aligned}
\end{equation*}
The results \eqref{eq:techn2:7} and \eqref{eq:techn2:16} for $T_{m}^{(2)}$ and $T_{m}^{(3)}$ respectively had the form
\begin{equation*}
\sum_{n\in T_{m}^{(2)}\cup T_{m}^{(3)}} R^{pd_n(V)}|a_n|^p \|f_n\|_p^p \lesssim_{p,R} (R\gamma)^{pm}\|f\|_p^p.
\end{equation*}
Observe that the $p$-norm of $f_n$ on the left hand side of the inequality is over the whole interval $[0,1]$. The same argument as for $T_m^{(2)}$ and $T_m^{(3)}$ yields 
\begin{equation}\label{eq:techn2:58}
\sum_{n\in T_{\operatorname{r},s}^{(2)}\cup T_{\operatorname{r},s}^{(3)}} R^{pd_n(V)}|a_n|^p \|f_n\|_p^p \lesssim_{p,R} (R\gamma)^{ps}\|f\|_p^p.
\end{equation}
Now, since
\begin{equation*}
\bigcup_{m\geq 0} T_{m}^{(6)}\subset \bigcup_{s\geq 0} T_{\operatorname{r},s}^{(2)}\cup T_{\operatorname{r},s}^{(3)},
\end{equation*}
inequality \eqref{eq:techn2:58} implies
\begin{equation}\label{eq:techn2:60}
\begin{aligned}
&\hspace{-0.5cm}\sum_{m=0}^\infty \sum_{n\in T_{m}^{(6)}} R^{pd_n(V)} |a_n|^p \|f_n\|_p^p  \\
&\leq \sum_{s=0}^\infty \sum_{n\in T_{r,s}^{(2)}\cup T_{r,s}^{(3)}} R^{p d_n(V)} |a_n|^p\|f_n\|_p^p\lesssim_{p,R} \|f\|_p^p
\end{aligned}
\end{equation}
After summation of \eqref{eq:techn2:4}, \eqref{eq:techn2:7}, \eqref{eq:techn2:16}, \eqref{eq:techn2:22} and \eqref{eq:techn2:54} over $m$, we add inequality \eqref{eq:techn2:60} to obtain finally
\begin{equation*}
\sum_{n\geq \polyfun(V)} R^{pd_n(V)} |a_n|^p \|f_n\|_{L^p(0,\widetilde{x})}^p \lesssim_{p,R} \|f\|_p^p,
\end{equation*}
The symmetric inequality
\begin{equation*}
\sum_{n\geq \polyfun(V)} R^{pd_n(V)} |a_n|^p \|f_n\|_{L^p(\widetilde{y},1)}^p \lesssim_{p,R} \|f\|_p^p
\end{equation*}
is treated analogously and thus, the proof of the lemma is completed.
\end{proof}

\section{Proof of the Main Theorem}\label{sec:main}
In this section, we prove our main result Theorem \ref{thm:uncond}, that is unconditionality of orthonormal spline systems corresponding to an arbitrary admissible point sequence $(t_n)_{n\geq 0}$ in reflexive $L^p$.
\begin{proof}[Proof of Theorem \ref{thm:uncond}]
We recall the notation
\[
Sf(t)=\Big(\sum_{n=-k+2}^\infty |a_nf_n(t)|^2\Big)^{1/2},\quad Mf(t)=\sup_{m\geq -k+2}\Big|\sum_{n=-k+2}^m a_n f_n(t)\Big|
\]
when
\[
f=\sum_{n=-k+2}^\infty a_nf_n.
\]
Since $(f_n)_{n=-k+2}^\infty$ is a basis in $L^p[0,1],\ 1\leq p<\infty,$ Khintchine's inequality implies that a necessary and sufficient condition for $(f_n)_{n=-k+2}^\infty$ to be an unconditional basis in $L^p[0,1]$ for some $p$ in the range $1<p<\infty$ is
\begin{equation}\label{eq:maintoprove}
\|Sf\|_p \sim_p \|f\|_p,\quad f\in L^p[0,1].
\end{equation}
We will prove \eqref{eq:maintoprove} for $1<p<2$ since the cases $p>2$ then follow by a duality argument.

We first prove the inequality 
\begin{equation}\label{eq:firsttoprove}
\|f\|_p\lesssim_p \|Sf\|_p.
\end{equation}
To begin with, let $f\in L^p[0,1]$ with $f=\sum_{n=-k+2}^\infty a_n f_n$. Without loss of generality, we may assume that the sequence $(a_n)_{n\geq -k+2}$ has only finitely many nonzero entries. We will prove \eqref{eq:firsttoprove} by showing the inequality $\|Mf\|_p\lesssim_p \|Sf\|_p$ and we first observe that 
\begin{equation}\label{eq:vertfkt}
\|Mf\|_p^p = p \int_0^\infty \lambda^{p-1} \psi(\lambda)\dif \lambda,
\end{equation}
with $\psi(\lambda):=[Mf>\lambda].$ Next we decompose $f$ into two parts $\varphi_1,\varphi_2$ and estimate the corresponding distribution functions $\psi_i(\lambda):=[M\varphi_i >\lambda/2]$, $i\in\{1,2\}$, separately. We continue with the definition of the functions $\varphi_i$. 
For $\lambda>0$, we define
\begin{align*}
E_\lambda &:= [Sf>\lambda],& B_\lambda&:=[\cM\charfun_{E_\lambda}>1/2], \\
\Gamma&:=\{n:J_n\subset B_\lambda,-k+2\leq n<\infty\},& \Lambda&:=\Gamma^c,
\end{align*}
where we recall that $J_n$ is the characteristic interval corresponding to the grid point $t_n$ and the function $f_n$.
Then, let
\begin{align*}
\varphi_1:= \sum_{n\in\Gamma} a_nf_n\qquad\text{and}\qquad \varphi_2:=\sum_{n\in\Lambda}a_nf_n.
\end{align*}
Now we estimate $\psi_1=[M\varphi_1>\lambda/2]$:
\begin{align*}
\psi_1(\lambda)
&=|\{t\in B_\lambda: M\varphi_1(t)>\lambda/2\}|+|\{t\notin B_\lambda: M\varphi_1(t)>\lambda/2\}| \\
&\leq |B_\lambda|+\frac{2}{\lambda}\int_{B_\lambda^c}M\varphi_1(t)\dif t \\
&\leq |B_\lambda|+\frac{2}{\lambda}\int_{B_\lambda^c} \sum_{n\in \Gamma} |a_nf_n(t)|\dif t.
\end{align*}
We decompose $B_\lambda$ into a disjoint collection of open subintervals of $[0,1]$ and apply Lemma \ref{lem:techn1} to each of those intervals to conclude from the latter expression
\begin{align*}
\psi_1(\lambda) &\lesssim |B_\lambda|+\frac{1}{\lambda}\int_{B_\lambda} Sf(t)\dif t \\
&= |B_\lambda|+\frac{1}{\lambda}\int_{B_\lambda\setminus E_\lambda} Sf(t)\dif t +\frac{1}{\lambda}\int_{E_\lambda\cap B_\lambda}Sf(t)\dif t \\
&\leq |B_\lambda|+|B_\lambda\setminus E_\lambda|+\frac{1}{\lambda}\int_{E_\lambda} Sf(t)\dif t,
\end{align*}
where in the last inequality, we simply used the definition of $E_\lambda$.
Since the Hardy-Littlewood maximal function operator $\mathcal M$ is of weak type (1,1), $|B_\lambda|\lesssim |E_\lambda|$ and thus we obtain finally
\begin{equation}
\label{eq:main:5}
\psi_1(\lambda)\lesssim |E_\lambda|+\frac{1}{\lambda}\int_{E_\lambda} Sf(t)\dif t.
\end{equation}
We now estimate $\psi_2(\lambda)$ and obtain from Theorem \ref{thm:maxbound} and the fact that $\mathcal M$ is a bounded operator on $L^2[0,1]$
\begin{equation*}
\begin{aligned}
\psi_2(\lambda)&\lesssim \frac{1}{\lambda^2}\|\cM \varphi_2\|_2^2\lesssim \frac{1}{\lambda^2}\|\varphi_2\|_2^2=\frac{1}{\lambda^2}\|S\varphi_2\|_2^2 \\
&=\frac{1}{\lambda^2}\Big(\int_{E_\lambda}S\varphi_2(t)^2\dif t+\int_{E_\lambda^c}S\varphi_2(t)^2\dif t\Big).
\end{aligned}
\end{equation*}
We apply Lemma \ref{lem:Sf} on the former expression to get
\begin{equation}\label{eq:main:3.5}
\psi_2(\lambda)\lesssim \frac{1}{\lambda^2}\int_{E_\lambda^c} S\varphi_2(t)^2\dif t
\end{equation}
Thus, combining \eqref{eq:main:5} and \eqref{eq:main:3.5},
\begin{equation*}
\begin{aligned}
\psi(\lambda)&\leq \psi_1(\lambda)+\psi_2(\lambda)\\
&\lesssim |E_\lambda|+\frac{1}{\lambda}\int_{E_\lambda}Sf(t)\dif t+\frac{1}{\lambda^2}\int_{E_\lambda^c}Sf(t)^2\dif t.
\end{aligned}
\end{equation*}
Inserting this inequality into \eqref{eq:vertfkt}, 
\begin{equation*}
\begin{aligned}
\|Mf\|_p^p
&\lesssim p\int_0^\infty \lambda^{p-1} |E_\lambda|\dif\lambda+p\int_0^\infty \lambda^{p-2}\int_{E_\lambda} Sf(t)\dif t\dif\lambda \\
&\quad+p\int_0^\infty \lambda^{p-3} \int_{E_\lambda^c}Sf(t)^2\dif t\dif\lambda \\
&=\|Sf\|_p^p+p\int_0^1 Sf(t)\int_0^{Sf(t)}\lambda^{p-2} \dif\lambda\dif t \\
&\quad + p\int_0^1 Sf(t)^2 \int_{Sf(t)}^\infty \lambda^{p-3}\dif\lambda\dif t,
\end{aligned}
\end{equation*}
and thus, since $1<p<2$, 
\[
\|Mf\|_p \lesssim_p \|Sf\|_p.
\]
So, the inequality $\|f\|_p\lesssim_p \|Sf\|_p$ is proved.

We now turn to the proof of the inequality
\begin{equation}\label{eq:mainproofsquarefunction}
\|Sf\|_p\lesssim_p \|f\|_p,\qquad 1<p<2.
\end{equation}
It is enough to show that the operator $S$ is of weak type $(p,p)$ for each exponent $p$ in the range $1<p<2$. This is because $S$ is (clearly) also of strong type $2$ and we can use the Marcinkiewicz interpolation theorem to obtain \eqref{eq:mainproofsquarefunction}.
Thus we have to show
\begin{equation}\label{eq:mainproofweaktypesquarefunction}
|[Sf>\lambda]|\lesssim_p \frac{\|f\|_p^p}{\lambda^p},\qquad f\in L^p[0,1],\ \lambda>0.
\end{equation}
We fix the function $f$ and the parameter $\lambda>0$. To begin with the proof of \eqref{eq:mainproofweaktypesquarefunction}, we define $G_\lambda:=[\mathcal Mf>\lambda]$ for $\lambda>0$ and observe that
\begin{equation}\label{eq:weaktypeGlambda}
|G_\lambda| \lesssim_p \frac{\|f\|_p^p}{\lambda^p},
\end{equation}
since $\mathcal M$ is of weak type $(p,p)$, and, by the Lebesgue differentiation theorem,
\begin{equation}\label{eq:lebesgue}
|f|\leq \lambda\qquad\text{a.\,e. on $G_\lambda^c$}.
\end{equation}
We decompose the open set $G_\lambda\subset[0,1]$ into a collection $(V_j)_{j=1}^\infty$ of disjoint open subintervals of $[0,1]$ and split the function $f$ into the two parts $h$ and $g$ defined by
\begin{equation*}
h:=f\cdot\charfun_{G_\lambda^c}+\sum_{j=1}^\infty T_{V_j}f,\qquad g:=f-h,
\end{equation*}
where for fixed index $j$, $T_{V_j}f$ is the projection of $f\cdot\charfun_{V_j}$ onto the space of polynomials of order $k$ on the interval $V_j$.

We treat the functions $h,g$ separately and begin with $h$. The definition of $h$ implies
\[
\|h\|_2^2=\int_{G_\lambda^c} |f(t)|^2\dif t+\sum_{j=1}^\infty \int_{V_j} (T_{V_j}f)(t)^2\dif t,
\]
since the intervals $V_j$ are disjoint. We apply \eqref{eq:lebesgue} to the first summand and \eqref{eq:polyproj1} to the second to obtain
\[
\|h\|_2^2 \lesssim \lambda^{2-p}\int_{G_\lambda^c} |f(t)|^p\dif t+\lambda^2 |G_\lambda|, 
\]
and thus, in view of \eqref{eq:weaktypeGlambda},
\[
\|h\|_2^2\lesssim_p \lambda^{2-p}\|f\|_p^p.
\]
This inequality allows us to estimate
\begin{equation*}
|[Sh>\lambda/2]|\leq \frac{4}{\lambda^2}\|Sh\|_2^2 = \frac{4}{\lambda^2}\|h\|_2^2\lesssim_p \frac{\|f\|_p^p}{\lambda^p},
\end{equation*}
which concludes the proof of \eqref{eq:mainproofweaktypesquarefunction} for the part $h$.

We turn to the proof of \eqref{eq:mainproofweaktypesquarefunction} for the function $g$. Since $p<2$, we have 
\begin{equation}\label{eq:main:8}
Sg(t)^p=\Big(\sum_{n=-k+2}^\infty|\langle g,f_n\rangle|^2 f_n(t)^2\Big)^{p/2}\leq \sum_{n=-k+2}^\infty |\langle g,f_n\rangle|^p |f_n(t)|^p
\end{equation}
For each index $j$, we define $\widetilde{V}_j$ to be the open interval with the same center as $V_j$ but with $5$ times its length. Then, set $\widetilde{G}_\lambda:=\bigcup_{j=1}^\infty \widetilde{V}_j\cap [0,1]$ and observe that $|\widetilde{G}_\lambda|\leq 5|G_\lambda|$. We get
\begin{equation*}
|[Sg>\lambda/2]|\leq |\widetilde{G}_\lambda|+\frac{2^p}{\lambda^p}\int_{\widetilde{G}_\lambda^c}Sg(t)^p\dif t.
\end{equation*}
By \eqref{eq:weaktypeGlambda}  and \eqref{eq:main:8}, this becomes
\[
|[Sg>\lambda/2]|\lesssim_p \lambda^{-p}\Big(\|f\|_p^p+\sum_{n=-k+2}^\infty \int_{\widetilde{G}_\lambda^c}|\langle g,f_n\rangle|^p|f_n(t)|^p\dif t\Big).
\]
But by definition of $g$ and \eqref{eq:polyproj2},
\[
\|g\|_p^p=\sum_j \int_{V_j} |f(t)-T_{V_j}f(t)|^p\dif t\lesssim_p\sum_j\int_{V_j}|f(t)|^p\lesssim \|f\|_p^p,
\]
so in order to prove the inequality $|[Sg>\lambda/2]|\leq \lambda^{-p}\|f\|_p^p$, it is enough to show the inequality
\begin{equation}\label{eq:main:finaltoprove}
\sum_{n=-k+2}^\infty \int_{\widetilde{G}_\lambda^c} |\langle g,f_n\rangle|^p|f_n(t)|^p\dif t\lesssim \|g\|_p^p.
\end{equation}
We now let $g_j:=g\cdot \charfun_{V_j}$. The supports of $g_j$ are therefore disjoint and we have $\|g\|_p^p=\sum_{j=1}^\infty \|g_j\|_p^p$. Furthermore $g=\sum_{j=1}^\infty g_j$ with convergence in $L^p$. Thus for each $n$, we obtain
\[
\langle g,f_n\rangle=\sum_{j=1}^\infty \langle g_j,f_n\rangle,
\]
and it follows from the definition of $g_j$ that
\begin{equation*}
\int_{V_j} g_j(t)p(t)\dif t=0
\end{equation*}
for each polynomial $p$ on $V_j$ of order $k$. This implies that $\langle g_j,f_n\rangle=0$ for $n<\polyfun(V_j)$, where 
\[
\polyfun(V):=\min\{n:\cT_n\cap V\neq\emptyset\}.
\]
Thus we obtain for all $R>1$ and for every $n$ 
\begin{equation}\label{eq:main:9}
\begin{aligned}
|\langle g,f_n\rangle |^p&=\Big| \sum_{j:n\geq\polyfun(V_j)}\langle g_j,f_n\rangle\Big|^p\leq \Big(\sum_{j:n\geq \polyfun(V_j)}R^{d_n(V_j)}|\langle g_j,f_n\rangle|R^{-d_n(V_j)}\Big)^p \\
&\leq \Big(\sum_{j:n\geq \polyfun(V_j)}R^{p d_n(V_j)}|\langle g_j,f_n\rangle|^p\Big)\Big(\sum_{j:n\geq\polyfun(V_j)}R^{-p'd_n(V_j)}\Big)^{p/p'},
\end{aligned}
\end{equation}
where $p'=p/(p-1)$.
If we fix $n\geq\polyfun(V_j)$, there is at least one point of the partition $\mathcal T_n$ contained in $V_j$. This implies that for each fixed $s\geq 0$, there are at most two indices $j$ such that $n\geq \polyfun(V_j)$ and $d_n(V_j)=s$. Therefore, 
\[
\Big(\sum_{j:n\geq\polyfun(V_j)}R^{-p'd_n(V_j)}\Big)^{p/p'}\lesssim_p 1,
\]
thus we obtain from \eqref{eq:main:9},
\begin{equation*}
|\langle g,f_n\rangle |^p\lesssim_p \sum_{j:n\geq\polyfun(V_j)}R^{pd_n(V_j)}|\langle g_j,f_n\rangle|^p.
\end{equation*}
Now we insert this inequality in \eqref{eq:main:finaltoprove} to get
\begin{equation*}
\begin{aligned}
\sum_{n=-k+2}^\infty&\int_{\widetilde{G}_\lambda^c}|\langle g,f_n\rangle|^p|f_n(t)|^p\dif t \\
&\lesssim_p\sum_{n=-k+2}^\infty\sum_{j:n\geq\polyfun(V_j)}R^{p d_n(V_j)}|\langle g_j,f_n\rangle|^p\int_{\widetilde{G}_\lambda^c} |f_n(t)|^p\dif t \\
&\leq \sum_{n=-k+2}^\infty\sum_{j:n\geq\polyfun(V_j)}R^{p d_n(V_j)}|\langle g_j,f_n\rangle|^p\int_{\widetilde{V}_j^c} |f_n(t)|^p\dif t \\
&\leq \sum_{j=1}^\infty\sum_{n\geq\polyfun(V_j)}R^{p d_n(V_j)}|\langle g_j,f_n\rangle|^p\int_{\widetilde{V}_j^c} |f_n(t)|^p\dif t 
\end{aligned}
\end{equation*}
We choose $R>1$ such that $R\gamma<1$ with the parameter $\gamma<1$ from Theorem \ref{thm:maintool} and apply Lemma \ref{lem:techn2} to obtain
\[
\sum_{n=-k+2}^\infty\int_{\widetilde{G}_\lambda^c}|\langle g,f_n\rangle|^p|f_n(t)|^p\dif t \lesssim_p \sum_{j=1}^\infty \|g_j\|_p^p = \|g\|_p^p,
\]
proving \eqref{eq:main:finaltoprove} and with it the inequality $\|Sf\|_p^p\lesssim_p \|f\|_p^p$. Thus the proof of Theorem \ref{thm:uncond} is completed.
\end{proof}

\subsubsection*{\bfseries \emph{Acknowledgments}}
During the development of this paper, the author visited repeatedly the IMPAN in Sopot/Gda\'nsk. It is his pleasure to thank this institution for its hospitality and for providing excellent working conditions. He is grateful to Anna\,Kamont, who suggested the problem considered in this paper and contributed many valuable remarks towards its solution, and to Zbigniew\,Ciesielski for many useful discussions. 

The author is supported by the Austrian Science Fund, FWF project P 23987-N18 and the stays in Sopot/Gda\'nsk were supported by MNiSW grant N N201 607840.

\bibliographystyle{plain}
\bibliography{uncondspline}
\end{document}